\documentclass[12pt,a4paper]{amsart}
\usepackage{fullpage}
\newcounter{mycount}
\numberwithin{mycount}{section}
\usepackage{amssymb,amsmath,amsthm}
\usepackage{tikz}
\usetikzlibrary{arrows,matrix,calc}
\usepackage{hyperref}

\newtheorem{thm}[mycount]{Theorem}
\newtheorem{cor}[mycount]{Corollary}
\newtheorem{lem}[mycount]{Lemma}
\newtheorem{prop}[mycount]{Proposition}
\theoremstyle{definition}
\newtheorem{dfn}[mycount]{Definition}
\theoremstyle{remark}
\newtheorem{rmk}[mycount]{Remark}

\newcommand{\ds}{\displaystyle}

\renewcommand{\O}{\mathcal O}
\newcommand{\D}{\mathcal D}
\newcommand{\E}{\mathcal E}
\newcommand{\F}{\mathcal F}
\newcommand{\G}{\mathcal G}

\newcommand{\I}{\mathcal I}
\newcommand{\J}{\mathcal J}
\renewcommand{\L}{\mathcal L}
\newcommand{\M}{\mathcal M}
\newcommand{\N}{\mathcal N}
\renewcommand{\P}{\mathcal P}
\newcommand{\W}{\mathcal W}
\newcommand{\Tot}{\mathrm{Tot}}
\renewcommand{\Bar}{\widehat{\mathcal Bar}}

\newcommand{\C}{\mathfrak C}
\newcommand{\Perf}{\mathrm{Perf}}
\newcommand{\Hom}{\mathrm{Hom}}
\newcommand{\sHom}{\mathcal{H}om}
\newcommand{\tens}[1]{\underset{#1}{\otimes}}
\newcommand{\ltens}[1]{\overset{\mathbb{L}}{%
\underset{#1}{\otimes}%
}%
}
\newcommand{\At}{\mathcal At}
\newcommand{\MF}{\mathfrak{MF}}
\newcommand{\Qcoh}{\mathfrak{Qcoh}}
\newcommand{\Coh}{\mathfrak{Coh}}

\newcommand{\RHom}{\mathbb{R}\mathrm{Hom}}

\DeclareMathOperator*{\hocolim}{hocolim}

\author{David Platt}
\date{\today}
\title{Chern Character For Global Matrix Factorizations}
\begin{document}
\begin{abstract}
We give a formula for the Chern character on the DG-category of global matrix factorizations on a smooth scheme $X$ with superpotential $w\in \Gamma(\O_X)$. Our formula takes values in a Cech model for Hochschild homology. Our methods may also be adapted to get an explicit formula for the Chern character for perfect complexes of sheaves on $X$ taking values in right derived global sections of the De-Rham algebra. Along the way we prove that the DG version of the Chern Character coincides with the classical one for perfect complexes. 
\end{abstract}

\maketitle

\tableofcontents

\section{Introduction} \label{section:intro}

Shklyarov in \cite{Shkly} gives a beautiful interpretation of the Chern Character and Riemann-Roch theorem in the context of DG-categories over a field $k$. In his treatment, he  uses functoriality of Hochschild homology and the canonical functor $k_E:k\to \mathfrak C$, which simply sends the DG-algebra $k$ to the object $E\in \mathfrak C$, to get the Chern character of $E$,
\[ch(E)=HH(k_E):k=\mathrm{HH}(k)\to \mathrm{HH}(\mathfrak C).\] 
In the case when $\mathfrak C$ is a proper DG-category, i.e. the diagonal bimodule, $\Delta$, takes values in perfect $k$-modules ($\Perf~k$) , we use the Kunneth isomorphism and the isomorphism $\mathrm{HH(\mathfrak{C}^{op})}\cong \mathrm{HH(\mathfrak C)}$ to obtain a pairing on homology:
\[<-,->_{\mathfrak C}:\mathrm{HH}(\mathfrak C)\otimes \mathrm{HH}(\mathfrak C)\cong \mathrm{HH}(\mathfrak C\otimes \mathfrak C^{op})\stackrel{\mathrm{HH(\Delta)}}{\to} \mathrm{HH}(\Perf k)=k\]
With this pairing and definition of the Chern character, the Riemann-Roch theorem,
\[<ch(E),ch(F)>_{\mathfrak C}=\mathrm{str}~\Hom_{\mathfrak C}(E,F),\] then becomes almost tautological, following easily from functoriality.

As with all beautiful things, the hard part is in the application. That is, for a particular DG-category, $\mathfrak C$, the difficulty is to get a meaningful handle on the Chern character and the pairing on Hochschild homology. The DG-categories of interest to us presently are certain categories of (global) matrix factorizations. We also only focus on the first half of the problem, i.e. to compute the Chern Character, taking values in some reasonable model for Hochschild homology.  
We, in fact, concern ourselves with a mildly more general problem: to compute the so called {\em boundary bulk map}. This is a map from the endomorphism DG-algebra of an object to Hochschild homology. We recommend \cite{PV:HRR} for details, in particular for the proof of the fact that the Chern Character is simply the boundary bulk map evaluated at the identity. Our formula for this map is rather involved, too much so to reproduce here (the impatient reader may thumb to theorem \ref{thm:formula}), however in the case when our matrix factorization, $\E$, admits a global connection, $\nabla$, i.e. global connections on graded components $\nabla_i:\E_i\to \Omega\otimes \E_i$, $i=0,1$, we obtain the following formula for the Chern Character:
\[ch(\E)=str\left(\sum_{i=0}^{\mathrm{dim~} X} \frac{[\nabla,e]^{i}}{i!}\right)\]
where $str$ denotes the super-trace, $e$ is the curved differential on $\E$, and 
\[[\nabla,e]=\nabla_{i+1} e_i-1\otimes e_i \nabla_{i}\qquad i=0,1.\]
We save understanding the pairing for a later work.

This paper is organized as follows. Section \ref{section:MF} contains the background information on our particular version of matrix factorizations (taken from \cite{PV:Stacks}). Section \ref{section:coderived} contains the DG/triangulated category theory pertaining to matrix factorizations that we will need. Some of results therein have not appeared in the generality in which we state them, but by no means is anything new. In section \ref{section:HH}, we carry out the computation of Hochschild Homology for our categories of matrix factorizations. The method for this computation is suggested in \cite{PL} and the analogous computation is carried out for Hochschild Cohomology therein. We give the details for homology. This result is also know by other methods from \cite{Preygel}. 

Sections \ref{section:atiyah} and \ref{section:formula} form the heart of the paper, culminating in the a formula for the boundary-bulk map which takes values in a Cech model for Hochschild Homology of matrix factorizations. This formula makes use of a choice of local connections on a Cech cover of the scheme $X$. In our opinion, more interesting than the formula, is the observation that the boundary-bulk map, which is a map in the derived category of complexes of vector spaces, may be promoted to the derived category of sheaves on our space $X$. Section \ref{section:atiyah} is concerned with understanding this promotion. Section \ref{section:formula} is concerned with what then happens upon applying right-derived global sections.

Throughout we assume the reader is mildly familiar with DG-categories and recommend \cite{Toen:Topics} for those who are not. Our specific conventions are as follows. We fix once and for all a field $k$. As one particular foundational lemma (\ref{lem:exteriorProduct}) will require it, we assume that $k$ is perfect. All categories we work with will be $k$-linear. In particular, ``scheme" will mean $k$-scheme so that quasi-coherent sheaves on said scheme form a $k$-linear category. $C(k)$ will denote the category of chain complexes of $k$ vector-spaces. By {\em modules} and $\mathfrak C-Mod$ we will always mean right modules, i.e. contravariant DG functors from $\mathfrak C$ with values in $C(k)$.  All grading will be $\mathbb Z$ gradings, though often there will be a 2-periodicity among graded components. We will use cohomological grading conventions. Subscripts will not denote a change from this convention, but instead will be used to reference internal grading for objects. For example if $C$ is a chain complex we will write $C_i$ for the $i$-th graded component (the differential has $d:C_{i}\to C_{i+1}$), whereas we would write
$\dots \to C^{-1}\to C^0\to C^1\to \dots$ 
for a complex of complexes with each $C^j$ a complex in its own right.  We will use $~^\#$ to denote the ``underlying graded object" functor, which forgets any additional structure (e.g. the differential) except the $\mathbb Z$-grading. For example if 
\[C=\dots\to C_{-1}\to C_0\to C_{1}\to \dots\]
is a chain complex then $C^{\#}$ is the graded vector space
\[C^{\#}=\bigoplus C_i.\]
For any DG category $\mathfrak C$, we may form two regular categories $Z^0\mathfrak C$ and $H^0\mathfrak C$, with the same objects as $\mathfrak C$ but with homs given by
\[\Hom_{Z^0\mathfrak C}(a,b)=Z^0\Hom_{\mathfrak C}(a,b)\quad \Hom_{H^0\mathfrak C}(a,b)=H^0\Hom_{\mathfrak C}(a,b)\]
where on the right-hand sides $Z^0$ and $H^0$ are the usual zero cycle and zero cohomology functors for chain complexes. We may also form the derived category $D\mathfrak C$, which is the Verdier localization of $H^0(\mathfrak C-Mod)$ with respect to quasi-isomorphisms.

\section{Matrix Factorizations}\label{section:MF}
We work with categories of matrix factorizations as in \cite{PV:Stacks} and \cite{Pos:CohMF}. Specifically we let $X$ be a noetherian $k$-scheme, $\L$ a line bundle on $X$ and $w\in \L(X)$ a global section. A {\em matrix factorization}, denoted 
\begin{center}
\begin{tikzpicture} 
\matrix(m)[matrix of math nodes, column sep=2em]{\E=\E_0 &\E_1\\};
\path[->]
(m-1-1.10) edge[bend left=15] node[auto] {$e_0$} (m-1-2)
(m-1-2) edge[bend right=-15] node[auto] {$e_1$} (m-1-1.base east);
\end{tikzpicture}
\end{center}
on $X$ with {\em potential} $w\in \L(X)$
consists of the data of two vector-bundles $\E_0$ and $\E_1$ on $X$ and maps 
\[e_1:\E_1\to \E_0 \quad \mbox{and} \quad  e_0:\E_0\to \E_1\otimes \L\]
such that $e_0e_1=id_{\E_1}\otimes w$ and $(e_1\otimes id_{\L})e_0=id_{\E_0}\otimes w$. Twisting by $\L$ and expanding 2-periodically we may view a matrix factorization as a ``complex" of sheaves except the differential, $e$, has $e^2$ is multiplication by $w$:
\begin{equation}\label{eqn:matfact}
\dots\to\E_{1}\otimes \L^{-1}\to \E_{0}\otimes \L^{-1}\to \E_1\to \E_0\to \E_1\otimes \L \to \E_0\otimes \L\to \dots.
\end{equation}
Here the term $\E_{i}\otimes \L^k$ lives in degree $2k-i$.
Given two matrix factorizations on $X$ with potential $w$, 
\begin{center}
\begin{tikzpicture} 
\matrix(m)[matrix of math nodes, column sep=2em]{\E=\E_0 &\E_1,\\};
\path[->]
(m-1-1.10) edge[bend left=15] node[auto] {$e_0$} (m-1-2)
(m-1-2) edge[bend right=-15] node[auto] {$e_1$} (m-1-1.base east);
\end{tikzpicture}
\begin{tikzpicture} 
\matrix(m)[matrix of math nodes, column sep=2em]{\D=\D_0 &\D_1\\};
\path[->]
(m-1-1.10) edge[bend left=15] node[auto] {$d_0$} (m-1-2)
(m-1-2) edge[bend right=-15] node[auto] {$d_1$} (m-1-1.base east);
\end{tikzpicture}
\end{center}
we may define a complex of morphisms $\Hom(\E,\D)$ whose underlying graded components are 
\[\Hom^{2k}(\E,\F):=\Hom_{\O_X}(\E_0,\F_0\otimes \L^{k})\oplus \Hom_{\O_X}(\E_1,\F_1\otimes \L^{k})\]
and
\[\Hom^{2k+1}(\E,\F):=\Hom_{\O_X}(\E_1,\F_0\otimes \L^{k})\oplus \Hom_{\O_X}(\E_0,\F_1\otimes \L^{k+1}).\]
The differential on $\Hom(\E,\D)$ is given by $\partial(f)=df-(-1)^{|f|}fe$. One easily verifies that $\partial^2=0$ so $\Hom(\E,\D)$ is indeed an honest complex, even though $\E$ and $\D$ are not.

The DG-category of matrix factorizations defined above is not ``correct" in the global setting. It contains objects which ``should be" 0 in the DG-derived category but are not, i.e. there are locally contractible matrix factorizations which are not globally contractible. There are several ways of dealing with this. In \cite{Orl:MF} Orlov defines the derived category of matrix factorizations to be the Verdier quotient of the derived category of matrix factorizations in standard DG sense by the thick subcategory of locally contractible objects. Alternatively one can form a new DG-category $\MF_{loc}(X,\L,w)$, in which we localize with respect to the spacial variable. The objects of $\MF_{loc}(X,\L,w)$ are the same as in $\MF(X,\L,w)$ and morphisms are given by a suitably functorial models (so that composition is well-defined) for the complexes $\RHom(\E_i,\F_j\otimes \L^{n})$ for $i,j=0,1$ and all $n$,
 then defining the morphism complex in the ``corrected" category to be
\[\Hom_{\MF_{loc}(X,\L,w)}(\E,\F):=\Tot(\RHom(\E,\F))\]
This can be done using a Cech model as in \cite{PL} and \cite{Shipman} or by choosing functorial injective resolutions which we explain below. 

In what follows we will want to consider a slightly larger class of objects obtained by dropping the restriction that $\E_0$ and $\E_1$ be vector bundles and allowing the graded components of $\E$ to arbitrary quasi-coherent sheaves on $X$. We will refer to such an object as a {\em curved quasi-coherent $\O_X$ module}. The DG-category of curved modules with Hom complexes defined above will be denoted by $\Qcoh(X,\L,w)$, the full subcategory of matrix factorizations will be denoted by $\MF(X,\L,w)$ and the full subcategory formed by considering curved sheaves with coherent graded components will be $\Coh(X,\L,w)$. We will often drop $X$ and or $\L$ from the notation, when they are clear from context. We will keep $w$, to distinguish $\Coh(w)$ (resp. $\Qcoh(w)$) from the categories $\mathbf{Coh}(X)$ (resp. $\mathbf{Qcoh}(X)$) of ordinary (quasi-)coherent sheaves on $X$. 

In \cite{Pos:TwoKinds} Positselski defines the the notion of a {\em curved differential graded ring} (CDG-ring) as a Graded ring $B=\bigoplus B_i$ along with a degree 1 endomorphism $d$ and an element $w\in B_2$ such that $\delta^2=[w,-]$. A $B$-module is a graded (left) $B^\#$-module $M$ 
endowed with its own differential $d_M$ satisfying the compatibility identity
\[d_M(am)=d(a)m-(-1)^{|a|}ad_M(m).\] 
Morphisms between curved modules are $B^{\#}$ module morphisms and are endowed with a differential in the standard way. As with matrix factorizations this differential produces a complex.

As in \cite{Pos:CohMF} we may use a sheafified version of curved modules to describe the category $\Qcoh(w)$. We define a sheaf of CDG-rings $S(\L)=\bigoplus_{i\in \mathbb Z} \L^i$ as the ``free algebra" on $\L$, graded such that $\L$ lives in degree 2 and we endow $S(\L)$ with the trivial differential. Then a curved $\O_X$ module with potential $w\in \L(X)$ is a Quasi-coherent CDG $S(\L)$-module, i.e. a curved quasi-coherent module as defined above gives rise to a $\mathbb Z$-graded $S(\L)$-module
\[\dots\stackrel{e_1}{\to} \E_{0}\otimes \L^{-1} \stackrel{e_0}{\to}\E_1\to \E_0\stackrel{e_1}{\to} \E_1\otimes \L^1\stackrel{e_0}{\to} \E_0\otimes \L^2\stackrel{e_1}{\to}\dots\]
with differential $e$ such that $e^2$ is multiplication by $w$. One can check that the morphisms in $\Qcoh(w)$ are precisely the morphisms of CDG $S(\L)$-modules, i.e. graded-morphisms on the underlying $\O_X$ modules which commute with the $S(\L)$ action. 

Conversely given a curved $S(\L)$-module $(\M,d_\M)$, the natural isomorphisms 
\[\L^{-1}\otimes \L\cong \O_X\cong \L\otimes \L^{-1}\] and the associativity of multiplication imply that any $S(\L)$-module, $\M$, must have isomorpishms 
\[\M_i\otimes \L\cong \M_{i+2}\] for all $i$. This gives an equivalence of categories between the $\Qcoh(w)$ and the category of CDG curved $S(\L)$ modules with curvature $w$. We will use both interpretations of $\Qcoh(X,\L,w)$ interchangeably. We will continue to use the notation $\Qcoh(X,\L,w)$ for both.

For curved modules $\M\in \Coh(X,\L,w)$ and $\N\in \Qcoh(X,\L,w')$ we may form the curved module 
$\sHom_{S(\L)}(\M,\N) \in \Qcoh(X,\L,w'-w)$ defined by 
\[\sHom_{S(\L)}(\M,\N)^{\#}=\sHom_{S(\L)^{\#}}(\M^{\#},{\N}^{\#})\] and whose differential is given by $\partial(f)=d_{M'}f-(-1)^{|f|}fd_{\N}$. In particular for $\M\in \Coh(X,\L,w)$ we have the dual module \[\M^{\vee}=\sHom_{S(\L)}(\M,S(\L))\in \Coh(X,\L,-w).\] For $\M\in \Qcoh(X,\L,w)$ and $\N\in \Qcoh(X,\L,w')$ we may form the tensor product $\M \tens{S(\L)} \N\in \Qcoh(X,\L,w+w')$ defined by 
\[(\M\tens{S(\L)} \N)^\#=\M^\#\tens{S(\L)^\#}\N^\#\]
and whose differential is given by $d_{\M\otimes \N}=d_\M\otimes 1+1\otimes d_{\N}$. More explicitly for 
\begin{center}
\begin{tikzpicture} 
\matrix(m)[matrix of math nodes, column sep=2em]{\M=\M_0 &\M_1,\\};
\path[->]
(m-1-1.10) edge[bend left=15] node[right=.3pt, above=.3pt] {$m_0$} (m-1-2)
(m-1-2) edge[bend right=-15] node[right=.3pt, below=.3pt] {$m_1$} (m-1-1.base east);
\end{tikzpicture}
\begin{tikzpicture} 
\matrix(m)[matrix of math nodes, column sep=2em]{\N=\N_0 &\N_1\\};
\path[->]
(m-1-1.10) edge[bend left=15] node[auto] {$n_0$} (m-1-2)
(m-1-2) edge[bend right=-15] node[auto] {$n_1$} (m-1-1.base east);
\end{tikzpicture}

\end{center} we have\\

$\sHom_{S(\L)}(\M,\N)=$

\begin{center}
\begin{tikzpicture} 
\matrix(m)[matrix of math nodes, column sep=4em]{|[left]|\sHom(\M_0,\N_0)\oplus \sHom(\M_1,\N_1)&|[right]| \sHom(\M_0,\N_1)\oplus\sHom(\M_1,\N_0)\\};
\path[->]
(m-1-1.north east) edge[bend left=15] node[auto] {$\left(\begin{smallmatrix} {-n_0}_{\ast}&{m_0}^\ast\\
{m_1}^\ast&{-n_1}_\ast\end{smallmatrix}\right)$} (m-1-2.north west)
(m-1-2.south west) edge[bend right=-15] node[auto] {$\left(\begin{smallmatrix} {n_1}_{\ast}&{m_1}^\ast\\
{m_0}^\ast&{n_0}_\ast\end{smallmatrix}\right) $} (m-1-1.south east);
\end{tikzpicture}
\end{center} 

and \\

$\M\tens{S(\L)} \N=$

\begin{center}
\begin{tikzpicture} 
\matrix(m)[matrix of math nodes, column sep=4em]{\M_0\otimes\N_0\oplus \M_1\otimes\N_1 & \M_0\otimes\N_1\oplus\M_1\otimes\N_0\\};
\path[->]
(m-1-1.north east) edge[bend left=15] node[auto] {$\left(\begin{smallmatrix} n_0\otimes 1&1\otimes m_0\\
1\otimes m_1&n_1\otimes 1\end{smallmatrix}\right)$} (m-1-2.north west)
(m-1-2.south west) edge[bend right=-15] node[auto] {$\left(\begin{smallmatrix} n_1\otimes 1&1\otimes m_1\\
1\otimes m_0&n_0\otimes 1\end{smallmatrix}\right) $} (m-1-1.south east);
\end{tikzpicture}
\end{center}
The following facts about the $\sHom_{S(\L)}$ and $\tens{S(\L)}$ functors are easily verified by sheafifying the natural isomorphisms that arise when $X$ is affine. 
\begin{prop} \label{prop:functors}Let $\M\in \Coh(X,\L,w)$ $\N\in \Coh(X,\L,v)$, $\P\in \Qcoh(X,\L,u)$ and $\E\in \MF(X,\L,t)$ and $\D\in \MF(X,\L,s)$ then
\begin{enumerate}
\item $\sHom_{S(\L)}(\M\tens{S(\L)} \N,\P)\cong \sHom_{S(\L)}(\M,\sHom_{S(\L)}(\N,\P))$ naturally as objects of $Z^0\Qcoh(X,\L,u-v-w)$.\\

\item $\sHom_{S(\L)}(\M\tens{S(\L)} \E,\N)\cong \sHom_{S(\L)}(\M, \E^\vee\tens{S(\L)}\N)\cong \sHom_{\S(L)}(\M,\N)\tens{S(\L)} \E^{\vee}$ naturally as objects of $Z^0\Qcoh(X,\L,v-w-t)$.\\

\item $\E^{\vee}\tens{S(\L)} \D \cong \sHom_{S(\L)}(\E,\D)$ as objects of $Z^0\Qcoh(X,\L,s-t)$.\\ 

\item $(\E^{\vee})^{\vee}\cong \E$ naturally in $Z^0\MF(X,\L,t)$ and the functor \[\vee: \MF(X,\L,t)^{op}\to \MF(X,\L,-t)\] is an equivalence of DG-categories. \\

\end{enumerate}
\end{prop} 

The category $Z^0\Qcoh(w)$ 
is easily seen to be an abelian category with arbitrary direct sums. From \cite{Pos:CohMF} the homotopy category $H^0\Qcoh(w)$ admits a triangulated structure with the obvious shift functor and in which distinguished triangles are isomorphic to triangles of the form
\[\E\stackrel{f}{\to} \mathcal D\to Cone(f)\to \E[1]\]
where $Cone(f)$ is defined analogously to cones in the category of complexes of sheaves. Also, as we will will use it frequently, if 
\[\dots \to \M^{-1}\to \M^0\to \M^1\to \dots\]
is a complex of curved modules (where the curved modules are viewed in the abelian category $Z^0\Qcoh(w)$ we may form the direct sum total curved module $\Tot(\M^{\bullet})$ whose graded components are $\Tot(\M^{\bullet})^{n}=\bigoplus_{p+q=n} \M^{p}_q$ and whose curved differential is given by the formula analogous to forming the total complex for complexes of sheaves. 

The functor $\#:Z^0\Qcoh(w)\to S(\L)^\#-Mod_0$, where $S(\L)^{\#}-Mod_0$ denotes the category of $S(\L)^{\#}$ modules with degree 0 morphisms, admits left and right adjoints $+$ and $-$ defined by
\begin{center}
\begin{tikzpicture} 
\matrix(m)[matrix of math nodes, column sep=4em]{\M^+=\M_0\oplus\M_1& \M_1\otimes \L^{-1}\oplus\M_0\\};
\path[->]
(m-1-1.north east) edge[bend left=15] node[auto] {$\left(\begin{smallmatrix} 0&1\\
w&0\end{smallmatrix}\right)$} (m-1-2.north west)
(m-1-2.south west) edge[bend right=-15] node[auto] {$\left(\begin{smallmatrix} 0&w\\
1&0\end{smallmatrix}\right) $} (m-1-1.south east);
\end{tikzpicture}
\end{center}
and 
\begin{center}
\begin{tikzpicture} 
\matrix(m)[matrix of math nodes, column sep=4em]{\M^-=\M_0\oplus\M_1\otimes \L& \M_1\oplus\M_0\\};
\path[->]
(m-1-1.north east) edge[bend left=15] node[auto] {$\left(\begin{smallmatrix} 0&w\\
1&0\end{smallmatrix}\right)$} (m-1-2.north west)
(m-1-2.south west) edge[bend right=-15] node[auto] {$\left(\begin{smallmatrix} 0&1\\
w&0\end{smallmatrix}\right) $} (m-1-1.south east);
\end{tikzpicture}
\end{center}
Evidently the functors $~^+$ and $~^-$ are exact.

We may use these adjoints to construct right and left resolutions in the abelian $Z^0\Qcoh(w)$, by first resolving as graded sheaves of $S(\L)^{\#}$-modules and then applying either $+$ or $-$ appropriately. Specifically when 
\[(\F^\bullet)^{\#}\to \E^\#\]
is a resolution of $\E$ as a graded $S(\L)$ module then 
\[((\F^\bullet)^{\#})^+\to \E\]
resolves $\E$ as a $w$-curved $S(\L)$ modules and similarly when 
\[\E^{\#}\to (\I^\bullet)^{\#}\]
resolves $\E$ then 
\[\E\to ((\I^\bullet)^{\#})^{-}\]
is a resolution as $w$-curved modules. We will be particularly interested in the cases when $(\F^\bullet)^\#$ consists of flat sheaves, vector bundles, or locally free sheaves and when $(I^\bullet)^\#$ consists of injective sheaves.

\section{The Coderived Category}\label{section:coderived}

We have yet to explain how we are to deal with locally contractible matrix factorizations or to justify our allegation that it is useful to pass to the larger category $\Qcoh(X,\L,w)$. 

As we mentioned above $H^0\Qcoh(w)$ is triangulated and this category, along with its triangulated structure, is reminiscent of homotopy category of complexes of quasi-coherent sheaves on $X$. As such, one is interested in localizing with respect to the ``acylic" objects, which would in particular kill the locally contractible matrix factorizations. The problem is that the usual notion of ``acyclic" has no obvious analog in $\Qcoh(X,\L,w)$ unless $w=0$. It turns out that the appropriate thing to do is to consider the exotic derived categories defined in \cite{Pos:TwoKinds}, in particular the so-called {\em coderived} category. 

\begin{dfn} We say that a curved module $\M\in H^0\Qcoh(X,\L,w)$ is {\em coacyclic}, if $\M$ is contained in the smallest triangulated category which contains the total curved modules 
\[\Tot(\mathcal A\to \mathcal B\to \mathcal C)\]
for all short exact sequences in $Z^0\Qcoh(w)$ and which is in addition closed under arbitrary direct sums. We will denote the triangulated category of coacyclic objects by $Coac(X,\L,w)$.
\end{dfn}

\begin{dfn} The coderived category of $\Qcoh(X,\L,w)$, denoted $D^{co}(\Qcoh(X,\L,w))$ is the Verdier quotient
\[D^{co}(\Qcoh(X,\L,w))=H^0\Qcoh(X,\L,w)/Coac(X,\L,w)\]
We will call the morphisms of $Z^0\Qcoh(X,\L,w)$ and $H^0 \Qcoh(X,\L,w)$ which become isomorphisms in the coderived categories weak equivalences. 
\end{dfn}

\begin{rmk}\label{rmk:resolutions}
One can easily check using induction that the coaccyclic modules contain the total curved modules of arbitrarily long, but finite, exact sequences. In particular a curved module $\E$ is weakly equivalent to any of its finite right or left resolutions. The following lemma and corollary show that $\E$ is in fact weakly equivalent to any of of its {\em infinite} right resolutions.
\end{rmk}

\begin{lem} The category $Coac(X, \L,w)$ is closed under taking directed homotopy colimits. 
\end{lem} 
\begin{proof} This is true of any triangulated category which is closed under taking direct sums, since the homotopy colimit is computed as the cone over a particular map between the direct sums of the objects.
\end{proof} 

\begin{cor} If 
\[\E\to \I_1\to \I_2 \dots \] is an exact sequence in $Z^0\Qcoh(X,\L,w)$ then the canonical map $\E\to \Tot(I_\bullet)$ is a weak equivalence. 
\end{cor}
\begin{proof} We can identify 
\[Cone(\E\to \Tot(\I^\bullet))=\Tot(\E\to I^\bullet)=\hocolim_n \Tot(\tau^n(\E\to I^\bullet))\]
where $\tau^n$ denotes the canonical truncation. Each $\Tot(\tau^n(\E\to \I^\bullet))$ is the total complex of a finite exact sequence of curved modules therefore is in $Coac(X,w,\L)$ (cf. Remark \ref{rmk:resolutions}) and then by the lemma the homotopy colimit is coacyclic as well.
\end{proof}

The above corollary and the discussion involving the existence of injective resolutions in $Z^0\Qcoh(X,\L,w)$ which concluded the previous section essentially give the following corollary.

\begin{cor} \label{cor:Coderived} For every $\M\in H^0\Qcoh(X,\L,w)$ there is a triangle 
\[\mathcal A\to \M\to \mathcal I\to \mathcal A[1]\]
where $\mathcal A$ is coacyclic and $\I$ is graded-injective. In particular the coderived category is equivalent to the homotopy category of graded-injective curved modules.
\end{cor} 
\begin{proof} The first claim follows from the existence of injective replacements. That the coderived category is equivalent to the homotopy category of injective modules will follow from the general theory of Verdier localization (see for example \cite{Pos:TwoKinds}) provided we can prove that $Hom(\mathcal C, \I)$ is acyclic whenever $\mathcal C$ is coacyclic and $\I$ is graded injective. For this it suffices to consider the case when $\mathcal C$ is the total curved module of a short exact sequence of curved modules, for which the statement is obvious.
\end{proof} 

Corollary \ref{cor:Coderived} tells us how to compute the homs in the coderived category: we pick some graded-injective replacements $\I$ of $\M$ and $\J$ of $\N$ then 
\[\Hom_{D^{co}\Qcoh(X,\L,w)}(\M,\N)=H^0\Hom_{\Qcoh(X,\L,w)}(\I,\J).\]

The category $\mathbf \Qcoh(X)$ of quasi-coherent sheaves on $X$ is a Grothendieck category and therefore admits functorial injective resolutions. Such a functor can be used to a functorial injective replacement (by simply taking the direct sum total complex) $K:\Qcoh(X,\L,w)\to Inj(X,\L,w)$, where $Inj(X,\L,w)$ denotes the full DG subcategory of $\Qcoh(X,\L,w)$ formed by the objects whose graded components are injective as $\O_X$ modules. We will denote by $\MF_{loc}(X,\L,w)$ the full subcategory of $Inj(X,\L,w)$ formed by (the images of) matrix factorizations. We will will often simply write $\RHom(\E,\F)$ for 
\[\Hom_{\MF_{loc}(X,\L,w)}(\E,\F)=\Hom_{\Qcoh(X,\L,w)}(K(\E),K(\F)).\]

We will make use of \cite{Pos:CohMF} Corollary 2.3(l):
\begin{prop}[Positselski] \label{Prop:Gens} The image of the category $\Coh(X,\L,w)$ forms a set of compact generators in $D^{co}\Qcoh(X,\L,w)$.
\end{prop}

This proposition has as an immediate corollary :
\begin{cor} \label{cor:gens} Suppose $X$ is smooth. The image of $\MF(X,\L,w)$ forms a set of compact generators in $D^{co}\Qcoh(X,\L,w)$. 
\end{cor}
\begin{proof} We need only to show that matrix factorizations generate since proposition \ref{Prop:Gens} already implies that they are compact. Since $X$ is smooth, for a coherent curved module $\M$ we may form a finite resolution of $\M^{\#}$ by a complex of sheaves whose graded components are vector bundles. Applying the $~^+$ functor yields a resolution of $\M$ by matrix factorizations. Then $\M$ is weakly equivalent to the matrix factorization obtained by taking the total curved module of this resolution (c.f. once more remark \ref{rmk:resolutions}).
\end{proof}

This corollary then gives us an important characterization of the category of modules for $\MF_{loc}(X,\L,w)$ and in particular justifies our claim that is useful to expand our view to the whole category of quasi-coherent curved modules:

\begin{thm} The functor $\M\mapsto \Hom(-,\M)|_{\MF_{loc}}$ induces a triangulated equivalence between between $D^{co}\Qcoh(X,\L,w)$ and $D(\MF_{loc}(X,\L,w))$
\end{thm}
\begin{proof} In light of corollary \ref{cor:gens}, this is an application of theorem 5.1 from \cite{Dyc}.
\end{proof}

The following theorem is announced in \cite{PL}, in the case when $\L=\O_X$. We include a few details to the proof.

\begin{thm} \label{thm:cptgen}Assume that $w$ is not a zero divisor, i.e. the map $\O_X\stackrel{w}{\to} \L$ is injective and that $X$ is smooth, then the category $\MF_{loc}(X,\L,w)$ has a compact generator. 
\end{thm} 
\begin{proof} We use the global version of Orlov's theorem given as the Main Theorem (2.7) from \cite{Pos:CohMF} to get an equivalence
\[\Coh(X,\L,w)/Coac(X,\L,w)\cap \Coh(X,\L,w)\cong D^{b}_{Sing}(X_0/X)\]
where $X_0$ is closed subscheme defined by $w=0$ and $D^b_{sing}(X_0/X)$ is the relative singularity category defined in {\em Loc. Cit}. As a piece of notation we will set
\[D^{abs}\Coh(X,\L,w):=\Coh(X,\L,w)/Coac(X,\L,w)\cap \Coh(X,\L,w).\] By Rouquier theorem 7.39 \cite{Rou} the bounded derived category of coherent sheaves on $X_0$ has a classical generator, $\mathcal G$. This classical generator then descends to a classical generator for the quotient $\D^{b}_{Sing}(X_0/X)$ and therefore gives a classical generator (which we will also call $\mathcal G$) for the category $D^{abs}(\Coh(X,\L,w))$. By corollary \ref{cor:gens} since $\mathcal G$ is coherent, there is a weak equivalence between $\mathcal G$ and some matrix factorization $\mathcal E_{\mathcal G}$. By corollary \ref{cor:gens}, $D^{abs}\Coh(X,\L,w)$ generates $D^{co}\Qcoh(X,\L,w)$, and therefore $\E_{\mathcal G}$ also generates $D^{co}\Qcoh(X,\L,w)$. Applying the injective replacement functor $K:\Qcoh(X,\L,w)\to Inj(X,\L,w)$, and using Corollary \ref{cor:Coderived} we get $K(\mathcal E_{\mathcal G})$ is a compact generator for $Inj(X,\L,w)$. By definition $K(\E_G)$ lies in $\MF_{loc}(X,\L,w)$ therefore is a compact generator for $\MF_{loc}(X,\L,w)$ as well. 
\end{proof}

\begin{lem} \label{lem:exteriorProduct} Assume $X$ is a smooth $k$-scheme, with $k$ a perfect field. Let $w\in \O_X$ and let $\tilde{w}$ denote the doubled potential $\tilde{w}=p_1^\ast(w)-p_2^\ast(w)$ on $X\times X$. The exterior product induces a quasi-equivalence
\[\MF_{loc}(X,\O_X,w)\otimes \MF_{loc}(X,\O_X,w)^{op}-Mod\cong \MF_{loc}(X\times X,\O_{X\times X},\tilde{w})-Mod\]
and under this isomorphism the diagonal bimodule corresponds the the diagonal curved module $\Delta_\ast S(\O_X)$.
\end{lem} 
\begin{proof} This follows from the same arguments as \cite{PL} theorem 3.4.
\end{proof}

We recall from \cite{PV:HRR} for a DG-category, $\C$, we have the trace functor \[Tr: D(\C\otimes C^{op}-Mod) \to D(C(k))\] given by $Tr(M):=M\tens{C\otimes C^{op}} \Delta$, where $\Delta$ is the diagonal bimodule $\Delta(a,b)=\Hom_{\C}(b,a)$. Then by \cite{Toen:Topics} Hochschild homology is computed as $Tr(\Delta)$.

\begin{lem} \label{lem:trace} The the isomorphism $D(\MF_{loc}(X\times X,\tilde{w}))\cong D(\MF_{loc}(x,w)\otimes \MF_{loc}(X,w)^{op})$ from Lemma \ref{lem:exteriorProduct} followed by the trace functor  is quasi-isomorphic to the functor 
\[\mathbb R\Gamma(\mathbb L\Delta^{\ast}-).\]

\end{lem}
\begin{proof} Both $Tr$ and $\mathbb R\Gamma(\mathbb L \Delta^\ast-)$ are triangulated functors from \[D(\MF(X\times X,\tilde{w})\] to $C(k)$ that commute with arbitrary direct sums, so it will suffice to check that they give the same result at the compact generator of $D(\MF(X\times X,\tilde{w}))$. For this we compute
\[\mathbb R\Gamma(\mathbb L\Delta^\ast\E\boxtimes \F^{\vee})=\mathbb R \Gamma(\E\otimes \F^{\vee})=\RHom(\F,\E).\]

\end{proof}

\section{Hochschild Homology}\label{section:HH}
In this section we compute the Hochschild homology of the category of matrix factorizations in the case when $\L=\O_X$. In fact, from now on all of our results will apply only to the case $\L=\O_X$, we save the more general case for later work. We will also assume now on that $X$ is smooth. We follow very closely the computation of Hochschild cohomology which appears in \cite{PL}. An alternative computation appears in \cite{Preygel} and at this point this result is well-known to the experts. We include our computation for completeness and since we will later have use to examine more closely the particular isomorphisms needed to compare the Hochschild homology to a certain complex involving forms on $X$. 

Following \cite{PL}, we define the {\em complete bar complex} $\Bar$.  This complex has graded components $\Bar_{-q}=(p_{1,q+2})_\ast\O_{\mathfrak X^{q+2}}$ for $q\geq 0$, where $\mathfrak X^k$ is the completion of \[X^k=X\times\dots \times X\] along the diagonal and $p_{1,q+2}:X\times \dots \times X\to X\times X$ projects to the first and last factor in the obvious way. To reduce clutter with our notation, we will hence forth simply write $\O_{\mathfrak X^k}$, rather than the push forward onto the first and last factor. The reader hopefully will keep in mind that $\O_{\mathfrak X^k}$ is actually viewed as a sheaf on $X\times X$.

 The differential,
\[b:\Bar_{-q}\to\Bar_{-q+1}\] is given locally by the standard formula for the bar differential:
\[b(a_0\boxtimes\dots \boxtimes a_{q+1})=\sum_{i=0}^{q}(-1)^ia_0\boxtimes\dots \boxtimes a_ia_{i+1}\boxtimes \dots \boxtimes a_{q+1}.\]
Here (and elsewhere) we use $\boxtimes$ to emphasize that this is an external tensor (i.e) only scalars commute with it as opposed to a tensor over $\O_X$. 
We introduce a new ``differential" of degree -1, $B_{w}$, defined locally by the equation
\[B_{w}(a_0\boxtimes\dots\boxtimes a_{q+1})=\sum_{i=0}^{q}(-1)^{i}a_0\boxtimes \dots \boxtimes a_i\boxtimes w\boxtimes a_{i+1}\boxtimes \dots a_{q+1}.\]

We now define the {\em curved complete bar complex}, $\Bar_{\tilde{w}}$, as follows. This will be an object of $\Qcoh(X\times X,\O_{X\times X},\tilde{w})$, where once again $\tilde{w}=p_1^\ast(w)-p_2^\ast(w)$ and $p_i:X\times X\to X$ are the standard projections. Again this follows \cite{PL}.

  We put
\[(\Bar_{\tilde{w}})_{q}=\bigoplus_{p\equiv q\mod 2} \Bar_{-p}.\]
The map $B_{w}$ may be viewed as a map of degree 1 in $\Bar_{\tilde{w}}$ by mapping the factor 
$(\Bar_{\tilde{w}})_p$ in $(\Bar_{\tilde{w}})_q$ to $(\Bar_{\tilde{w}})_{-(p+1)}$ in $(\Bar_{\tilde {w}})_{q+1}$. We imbue $\Bar_{\tilde{w}}$ 
 with the curved differential $\partial=b+B_{w}$, then one checks
 that $B_{w}^2=0$ and then that
 \[\partial^2=bB_w+B_wb=\tilde{w}\]
 so $\Bar_{\tilde w}$ is indeed a $\tilde{w}$-curved module. 
 
 It is helpful to view $\Bar_{\tilde{w}}$ as the total complex (perhaps modulo some signs) of the following ``bi-complex":

 \begin{center}
 \begin{tikzpicture}[
cross line/.style={preaction={draw=white, -,
line width=6pt}},
descr/.style={fill=white,inner sep=2.5pt}]

 \matrix(m)[matrix of math nodes, row sep=2em, column sep=3em]
 {{}&\vdots&\vdots&\vdots\\
 \dots&0&0&0&\\
 \dots&\O_{\mathfrak X^4}&\O_{\mathfrak X^3}&\O_{\mathfrak X^2}\\
 \dots&0&0&0\\
 \dots&\O_{\mathfrak X^4}&\O_{\mathfrak X^3}&\O_{\mathfrak X^2}\\
\dots&0&0&0\\
&\vdots&\vdots&\vdots\\};

\path[->]
(m-2-1) edge[densely dotted] node[auto]{} (m-2-2)
(m-2-2) edge[densely dotted] node[auto]{} (m-2-3)
(m-2-3) edge[densely dotted] node[auto]{} (m-2-4)
(m-3-1) edge node[auto]{$b$} (m-3-2)
(m-3-2) edge node[auto]{$b$} (m-3-3)
(m-3-3) edge node[auto]{$b$} (m-3-4)
(m-4-1) edge[densely dotted] node[auto]{} (m-4-2)
(m-4-2) edge[densely dotted] node[auto]{} (m-4-3)
(m-4-3) edge[densely dotted] node[auto]{} (m-4-4)
(m-5-1) edge node[auto]{$b$} (m-5-2)
(m-5-2) edge node[auto]{$b$} (m-5-3)
(m-5-3) edge node[auto]{$b$} (m-5-4)
(m-6-1) edge[densely dotted] node[auto]{} (m-6-2)
(m-6-2) edge[densely dotted] node[auto]{} (m-6-3)
(m-6-3) edge[densely dotted] node[auto]{} (m-6-4)
(m-3-3) edge[cross line] node[left=16pt,above,sloped]{$B_{{w}}$} (m-1-2)
(m-3-4) edge[cross line] node[left=16pt,above,sloped]{$B_{{w}}$} (m-1-3)
(m-3-2) edge[cross line] node[left=16pt,above,sloped]{$B_{{w}}$} (m-1-1)
(m-7-3) edge[cross line] node[left=16pt,above,sloped]{$B_{{w}}$} (m-5-2)
(m-7-4) edge[cross line] node[left=16pt,above,sloped]{$B_{{w}}$} (m-5-3)
(m-7-2) edge[cross line] node[left=16pt,above,sloped]{$B_{{w}}$} (m-5-1)
(m-5-3) edge[cross line] node[left=16pt,above,sloped]{$B_{{w}}$} (m-3-2)
(m-5-4) edge[cross line] node[left=16pt,above,sloped]{$B_{{w}}$} (m-3-3)
(m-5-2) edge[cross line] node[left=16pt,above,sloped]{$B_{{w}}$} (m-3-1);
\end{tikzpicture}
\end{center}

There is a map $\Bar_{\tilde{w}}(X)\stackrel{\epsilon}{\to} \Delta_\ast S(\O_X)$ given by projecting the even components onto $\O_{\mathfrak X^2}$ then using the multiplication map 
\[\O_{\mathfrak X^2}\to \O_\Delta\]
and sending the odd components to 0. It is easy to check that this defines a closed degree 0 morphism of curved complexes.

\begin{lem} \label{lem:deltares} $\Bar_{\tilde{w}}\tens{S(\O_{X\times X})} \M$ is isomorphic to $\Delta\ltens{S(\O_{X\times X})}\M$ in $D(X)$ for any $-\tilde{w}$ curved module $\M$.
\end{lem}
\begin{proof} We let $\mathcal W$ be the cone of the morphism $\epsilon:\Bar_{\tilde{w}}(X)\to \Delta_\ast S(\O_X)$. Then
\[\mathcal W_n=\bigoplus_{k\equiv n\mod 2} \O_{\mathfrak X^k}\]
where we consider $\mathfrak X^1=\Delta$. 

Consider first the case when $\M$ is graded-flat. The $n$-th graded component of the complex $\mathcal W\tens{S(\O_{X\times X})} \M$ is
\begin{align*}
(\mathcal W\tens{S(\O_{X\times X})} \M)_n&=\bigoplus_{k\equiv n\mod 2}\O_{\mathfrak X^k}\otimes \M_0\oplus \bigoplus_{k\equiv n+1\mod 2}\O_{\mathfrak X^k}\otimes \M_1\\
&=\bigoplus_{k}\O_{\mathfrak X^k}\otimes \M_{n-k}\\
\end{align*} 
Taking the differential into account, may view $\mathcal W\tens{S(\O_{X\times X})} \M$ as the total complex of the ``bi-complex"
\begin{center}
 \begin{tikzpicture}[
cross line/.style={preaction={draw=white, -,
line width=6pt}},
descr/.style={fill=white,inner sep=2.5pt}]

 \matrix(m)[matrix of math nodes, row sep=2em, column sep=3em]
 {{~}&\vdots&\vdots&\vdots&\vdots\\
 \dots&\O_{\mathfrak X^4}\otimes \M_1&\O_{\mathfrak X^3}\otimes \M_1&\O_{\mathfrak X^2}\otimes \M_1&\O_{\Delta}\otimes \M_1\\
 \dots&\O_{\mathfrak X^4}\otimes \M_0&\O_{\mathfrak X^3}\otimes \M_0&\O_{\mathfrak X^2}\otimes \M_0&\O_{\Delta}\otimes \M_0\\
\dots&\O_{\mathfrak X^4}\otimes \M_1&\O_{\mathfrak X^3}\otimes \M_1&\O_{\mathfrak X^2}\otimes \M_1 &\O_{\Delta}\otimes \M_1\\
 \dots&\O_{\mathfrak X^4}\otimes \M_0&\O_{\mathfrak X^3}\otimes \M_0&\O_{\mathfrak X^2}\otimes \M_0&\O_{\Delta}\otimes \M_0\\
 \dots&\O_{\mathfrak X^4}\otimes \M_1&\O_{\mathfrak X^3}\otimes \M_1&\O_{\mathfrak X^2}\otimes \M_1&\O_{\Delta}\otimes \M_1\\
{}&\vdots&\vdots&\vdots&\vdots\\};

\path[->]
(m-2-1) edge[cross line] node[auto]{} (m-2-2)
(m-2-2) edge[cross line] node[auto]{} (m-2-3)
(m-2-3) edge[cross line] node[auto]{} (m-2-4)
(m-3-1) edge[cross line] node[auto]{} (m-3-2)
(m-3-2) edge[cross line] node[auto]{} (m-3-3)
(m-3-3) edge[cross line] node[auto]{} (m-3-4)
(m-4-1) edge[cross line] node[auto]{} (m-4-2)
(m-4-2) edge[cross line] node[auto]{} (m-4-3)
(m-4-3) edge[cross line] node[auto]{} (m-4-4)
(m-5-1) edge[cross line] node[auto]{} (m-5-2)
(m-5-2) edge[cross line] node[auto]{} (m-5-3)
(m-5-3) edge[cross line] node[auto]{} (m-5-4)
(m-6-1) edge[cross line] node[auto]{} (m-6-2)
(m-6-2) edge[cross line] node[auto]{} (m-6-3)
(m-6-3) edge[cross line] node[auto]{} (m-6-4)
(m-2-4) edge node[auto]{$\epsilon$} (m-2-5)
(m-3-4) edge node[auto]{$\epsilon$} (m-3-5)
(m-4-4) edge node[auto]{$\epsilon$} (m-4-5)
(m-5-4) edge node[auto]{$\epsilon$} (m-5-5)
(m-6-4) edge node[auto]{$\epsilon$} (m-6-5)
(m-1-2) edge node[auto]{} (m-1-2)
(m-3-2) edge node[auto]{} (m-2-2)
(m-4-2) edge node[auto]{} (m-3-2)
(m-5-2) edge node[auto]{} (m-4-2)
(m-6-2) edge node[auto]{} (m-5-2)
(m-7-2) edge node[auto]{} (m-6-2)
(m-2-3) edge node[auto]{} (m-1-3)
(m-3-3) edge node[auto]{} (m-2-3)
(m-4-3) edge node[auto]{} (m-3-3)
(m-5-3) edge node[auto]{} (m-4-3)
(m-6-3) edge node[auto]{} (m-5-3)
(m-7-3) edge node[auto]{} (m-6-3)
(m-2-4) edge node[auto]{} (m-1-4)
(m-3-4) edge node[auto]{} (m-2-4)
(m-4-4) edge node[auto]{} (m-3-4)
(m-5-4) edge node[auto]{} (m-4-4)
(m-6-4) edge node[auto]{} (m-5-4)
(m-7-4) edge node[auto]{} (m-6-4)
(m-2-5) edge node[auto]{} (m-1-5)
(m-3-5) edge node[auto]{} (m-2-5)
(m-4-5) edge node[auto]{} (m-3-5)
(m-5-5) edge node[auto]{} (m-4-5)
(m-6-5) edge node[auto]{} (m-5-5)
(m-7-5) edge node[auto]{} (m-6-5)
(m-3-3) edge[densely dotted] node[left=16pt,above]{} (m-1-2)
(m-3-4) edge[densely dotted] node[left=16pt,above]{} (m-1-3)
(m-3-2) edge[densely dotted] node[left=16pt,above,sloped]{} (m-1-1)
(m-7-3) edge[densely dotted] node[left=16pt,above,sloped]{} (m-5-2)
(m-7-4) edge[densely dotted] node[left=16pt,above,sloped]{} (m-5-3)
(m-7-2) edge[densely dotted] node[left=16pt,above,sloped]{} (m-5-1)
(m-5-3) edge[densely dotted] node[left=16pt,above,sloped]{} (m-3-2)
(m-5-4) edge[densely dotted] node[left=16pt,above,sloped]{} (m-3-3)
(m-5-2) edge[densely dotted] node[left=16pt,above,sloped]{} (m-3-1)
(m-4-3) edge[densely dotted] node[left=16pt,above]{} (m-2-2)
(m-4-4) edge[densely dotted] node[left=16pt,above]{} (m-2-3)
(m-4-2) edge[densely dotted] node[left=16pt,above,sloped]{} (m-2-1)
(m-6-3) edge[densely dotted] node[left=16pt,above,sloped]{} (m-4-2)
(m-6-4) edge[densely dotted] node[left=16pt,above,sloped]{} (m-4-3)
(m-6-2) edge[densely dotted] node[left=16pt,above,sloped]{} (m-4-1);
\end{tikzpicture}
\end{center}
 where the horizontal maps are induced by $b$, diagonal maps induced by $B_{w}$ and the vertical maps by the differential on $\M$. This "bi-complex" is of course just a mnemonic, but it gives us insight into how to deal with the complex $\mathcal W\tens{S(\O_{X\times X})} \M$. In particular, we may ``filter the bi-complex by rows" to get a filtration on $\mathcal W\tens{S(\O_{X\times X})} \M$. One should convince oneself that this indeed a filtration by subcomplexe. This filtration is bounded below and exhaustive, therefore the associated spectral sequence converges. Already on the $E_1$ page all of the groups are 0 since the rows of the ```bi-complex" associated to $\mathcal W\tens{S(\O_{X\times X})} \M$ are exact. This gives us that the map
 \[{\Bar}_{\tilde{w}}(X)\tens{S(\O_{X\times X})} \M\to \Delta_\ast S(\O_X)\otimes_{S(\O_{X\times X})} \M\]
 is a quasi-isomorphism.
 
 Now for general $\mathcal M$, let $\underline{\M}=\Tot(\mathcal F^{\bullet})$ be a flat replacement of $\M$, where $\F^{\bullet}$ is a (finite) resolution of $\M$ by flat $-\tilde{w}$-curved modules. This can be done as in corollary \ref{cor:gens}; finiteness is possible since $X$ is smooth. It is well known (see for example \cite{Yek}) that $(p_{1,k})_\ast{\O_{\mathfrak X^k}}$ is flat as an $\O_{X^2}$-module and therefore the graded components of  $\Bar_{\tilde{w}}$ are flat. This implies the morphism 
\[\Bar_{\tilde{w}}\otimes \underline{\M}\to\Bar_{\tilde{w}}\otimes \M\]
is a quasi-isomorphism. 

We have
\[\Bar_{\tilde{w}}\tens{S(\O_{X\times X})} \underline{\M}=\Bar_{\tilde{w}}\tens{S(\O_{X\times X})} \Tot(\F^{\bullet})=\Tot(\Bar_{\tilde{w}}\tens{S(\O_{X\times X})} \F^{\bullet})\]
The cone of the morphism 
\[\Tot(\Bar_{\tilde{w}}\otimes \F^{\bullet})\to \Tot(\Delta_\ast S(\O_X)\otimes \F^{\bullet})\]
is given by $\Tot(\W\otimes \F^{\bullet})$, which is the total complex of a bicomplex with exact columns (by the above argument) and uniformly bounded rows and therefore is acyclic. \\

Therefore we obtain a zig-zag of quasi-isomorphisms
\[\Bar_{\tilde{w}}\tens{S(\O_{X\times X})}  \M\leftarrow \Bar_{\tilde{w}}\tens{S(\O_{X\times X})} \underline{\M}\to  \Delta_\ast(S(\O_X))\tens{S(\O_{X\times X})} \underline{\M}\]
Since $\Delta_\ast(S(O_X))\tens{S(\O_{X\times X})} \underline{\M}$ computes $\Delta_{\ast} S(\O_X)\ltens{S(\O_{X\times X})} \M$, we are done. 
\end{proof}

\begin{lem} The map $\Bar_{\tilde w}\to \Delta_\ast S(\O_X)$ is a weak equivalence in $Z^0\Qcoh(X,\O_{X\times X},\tilde{w})$
\end{lem}
\begin{proof} Again we use $\mathcal W$ for the cone of the map $\Bar_{\tilde w} \to \Delta_\ast S(\O_X)$. Let $\mathcal G$ be a compact generator for $\Qcoh(X\times X,\O_{X\times X},\tilde{w})$ and by \ref{cor:gens} we can take $\mathcal G$ to be a matrix factorization. By the previous lemma \[\mathcal G^{\vee}\tens{S(\O_{X\times X})} \mathcal W=\sHom_{S(\O_{X\times X})}(\mathcal G,\mathcal W)\] is acyclic, and since $\mathcal G$ is locally free we have a quasi-isomorphism 
\[\sHom_{S(\O_{X\times X})}(\mathcal G,\mathcal W)\cong \sHom_{S(\O_{X\times X})}(\mathcal G,K(\mathcal W)),\]
where $K: \Qcoh(X\times ,\O_{X\times X},\tilde{w})\to Inj(X\times X,\O_{X\times X},\tilde{w})$ is our chosen functorial injective replacement. By adjunction, the complex of sheaves $\sHom_{S(\O_X)}(\mathcal G,K(\mathcal W))$ has injective graded components. We are want to say that that having injective graded components is sufficient for $\sHom(\mathcal G,K(\mathcal W))$ to be adapted to the global sections functor, if we could the proof would be done. However this complex is unbounded in both directions so care must be taken. 

Since $X$ is smooth thus has finite homological dimension, each of the cokernels of the differentials are injective. Then by exactness, the kernels of the differentials are also injective. Using these facts one can easily verify directly that the global sections functor is exact by checking at any particular spot and truncating appropriately, so that the truncated sequence is a bounded exact sequence of injective sheaves. Finally we can conclude that the complex of vector spaces 
\[\Hom(\mathcal G, K(\mathcal W))=\Gamma(\sHom(\mathcal G, K(\mathcal W)))\]
is exact. Since $\mathcal G$ is a generator this implies that $K(\mathcal W)$ is coacyclic and therefore $\mathcal W$ is as well.  
\end{proof} 

\begin{thm} The Hochschild homology of $\MF(X,\O_X,w)$ is $\mathbb R\Gamma(\Omega_{dw})$, where $\Omega_{dw}$ is the two periodic complex of sheaves
\begin{center}
\begin{tikzpicture}
\matrix(m)[matrix of math nodes, row sep=3em, column sep =3em]{%
\dots&\ds\bigoplus_{i~\mathrm{odd}}\Omega^i&\ds\bigoplus_{i~\mathrm{even}}\Omega^i&\ds\bigoplus_{i~\mathrm{odd}}\Omega^i&\ds\bigoplus_{i~\mathrm{even}}\Omega^i&\dots\\
};

\path[->]
(m-1-1) edge node[auto]{$dw\wedge$} (m-1-2)
(m-1-2) edge node[auto]{$dw\wedge$} (m-1-3)
(m-1-3) edge node[auto]{$dw\wedge$} (m-1-4)
(m-1-4) edge node[auto]{$dw\wedge$} (m-1-5)
(m-1-5) edge node[auto]{$dw\wedge$} (m-1-6);

\end{tikzpicture}
\end{center}
with $\displaystyle \bigoplus_{i~\mathrm{even}}\Omega^i$ in even degrees.
\end{thm}

\begin{proof} By lemmas \ref{lem:trace} and \ref{lem:exteriorProduct} we compute the hochschild homology of $\MF(X,\O_X,w)$ as $\mathbb R \Gamma(\mathbb L\Delta^\ast \Delta_\ast S(\O_X))$. By lemma \ref{lem:deltares}, we may compute $\mathbb L\Delta^\ast\Delta_{\ast}S(O_X)$ as 
$\Delta^{\ast}\Bar_{\tilde{w}}$. Now,
$\Delta^{\ast}\Bar_{\tilde{w}}$ is given as the total complex of the ``bi-complex"
\begin{equation}\label{eqn:HHbicomp}
 \begin{tikzpicture}[
cross line/.style={preaction={draw=white, -,
line width=6pt}},
descr/.style={fill=white,inner sep=2.5pt}]

 \matrix(m)[matrix of math nodes, row sep=2em, column sep=3em]
 {{~}&\vdots&\vdots&\vdots\\
 \dots&0&0&0\\
 \dots&\Delta^\ast\O_{\mathfrak X^4}&\Delta^\ast\O_{\mathfrak X^3}&\Delta^\ast\O_{\mathfrak X^2}\\
 \dots&0&0&0\\
 \dots&\Delta^\ast\O_{\mathfrak X^4}&\Delta^\ast\O_{\mathfrak X^3}&\Delta^\ast\O_{\mathfrak X^2}\\  \dots&0&0&0\\
 \dots&\Delta^\ast\O_{\mathfrak X^4}&\Delta^\ast\O_{\mathfrak X^3}&\Delta^\ast\O_{\mathfrak X^2}\\
{}&\vdots&\vdots&\vdots\\};

\path[->]
(m-2-1) edge[densely dotted] node[auto]{} (m-2-2)
(m-2-2) edge[densely dotted] node[auto]{} (m-2-3)
(m-2-3) edge[densely dotted] node[auto]{} (m-2-4)
(m-3-1) edge node[auto]{$b$} (m-3-2)
(m-3-2) edge node[auto]{$b$} (m-3-3)
(m-3-3) edge node[auto]{$b$} (m-3-4)
(m-4-1) edge[densely dotted] node[auto]{} (m-4-2)
(m-4-2) edge[densely dotted] node[auto]{} (m-4-3)
(m-4-3) edge[densely dotted] node[auto]{} (m-4-4)
(m-5-1) edge node[auto]{$b$} (m-5-2)
(m-5-2) edge node[auto]{$b$} (m-5-3)
(m-5-3) edge node[auto]{$b$} (m-5-4)
(m-6-1) edge[densely dotted] node[auto]{} (m-6-2)
(m-6-2) edge[densely dotted] node[auto]{} (m-6-3)
(m-6-3) edge[densely dotted] node[auto]{} (m-6-4)
(m-3-3) edge[cross line] node[left=16pt,above,sloped]{$B_{{w}}$} (m-1-2)
(m-3-4) edge[cross line] node[left=16pt,above,sloped]{$B_{{w}}$} (m-1-3)
(m-3-2) edge[cross line] node[left=16pt,above,sloped]{$B_{{w}}$} (m-1-1)
(m-7-3) edge[cross line] node[left=16pt,above,sloped]{$B_{{w}}$} (m-5-2)
(m-7-4) edge[cross line] node[left=16pt,above,sloped]{$B_{{w}}$} (m-5-3)
(m-7-2) edge[cross line] node[left=16pt,above,sloped]{$B_{{w}}$} (m-5-1)
(m-5-3) edge[cross line] node[left=16pt,above,sloped]{$B_{{w}}$} (m-3-2)
(m-5-4) edge[cross line] node[left=16pt,above,sloped]{$B_{{w}}$} (m-3-3)
(m-5-2) edge[cross line] node[left=16pt,above,sloped]{$B_{{w}}$} (m-3-1);
\end{tikzpicture}
\end{equation}

Applying the Hochshild-Kostant-Rosenburg (HKR) quasi-isomorphism, which is given locally by
\[a_0\boxtimes \dots \boxtimes a_q\mapsto \frac{1}{q!} a_0a_{q}da_1\wedge\dots \wedge da_{q-1}\] (see \cite{Yek} theorem 4.8) along the rows we obtain a quasi-isomorphism between \eqref{eqn:HHbicomp} and the bicomplex
\begin{equation} \label{eqn:HHbicomp2}
\begin{tikzpicture}[
cross line/.style={preaction={draw=white, -,
line width=6pt}},
descr/.style={fill=white,inner sep=2.5pt}]

 \matrix(m)[matrix of math nodes, row sep=2em, column sep=3em]
 {{~}&\vdots&\vdots&\vdots\\
 \dots&0&0&0\\
 \dots&\Omega^2&\Omega&\O_X\\
 \dots&0&0&0\\
 \dots&\Omega^2&\Omega&\O_X\\
   \dots&0&0&0\\
\dots&\Omega^2&\Omega&\O_X\\
{}&\vdots&\vdots&\vdots\\};

\path[->]
(m-2-1) edge[densely dotted] node[auto]{} (m-2-2)
(m-2-2) edge[densely dotted] node[auto]{} (m-2-3)
(m-2-3) edge[densely dotted] node[auto]{} (m-2-4)
(m-3-1) edge node[auto]{$0$} (m-3-2)
(m-3-2) edge node[auto]{$0$} (m-3-3)
(m-3-3) edge node[auto]{$0$} (m-3-4)
(m-4-1) edge[densely dotted] node[auto]{} (m-4-2)
(m-4-2) edge[densely dotted] node[auto]{} (m-4-3)
(m-4-3) edge[densely dotted] node[auto]{} (m-4-4)
(m-5-1) edge node[auto]{$0$} (m-5-2)
(m-5-2) edge node[auto]{$0$} (m-5-3)
(m-5-3) edge node[auto]{$0$} (m-5-4)
(m-7-1) edge node[auto]{$0$} (m-7-2)
(m-7-2) edge node[auto]{$0$} (m-7-3)
(m-7-3) edge node[auto]{$0$} (m-7-4)
(m-6-1) edge[densely dotted] node[auto]{} (m-6-2)
(m-6-2) edge[densely dotted] node[auto]{} (m-6-3)
(m-6-3) edge[densely dotted] node[auto]{} (m-6-4)
(m-3-3) edge[cross line] node[left=16pt,above,sloped]{$\wedge dw$} (m-1-2)
(m-3-4) edge[cross line] node[left=16pt,above,sloped]{$\wedge dw$} (m-1-3)
(m-3-2) edge[cross line] node[left=16pt,above,sloped]{$\wedge dw$} (m-1-1)
(m-7-3) edge[cross line] node[left=16pt,above,sloped]{$\wedge dw$} (m-5-2)
(m-7-4) edge[cross line] node[left=16pt,above,sloped]{$\wedge dw$} (m-5-3)
(m-7-2) edge[cross line] node[left=16pt,above,sloped]{$\wedge dw$} (m-5-1)
(m-5-3) edge[cross line] node[left=16pt,above,sloped]{$\wedge dw$} (m-3-2)
(m-5-4) edge[cross line] node[left=16pt,above,sloped]{$\wedge dw$} (m-3-3)
(m-5-2) edge[cross line] node[left=16pt,above,sloped]{$\wedge dw$} (m-3-1);
\end{tikzpicture}
\end{equation}
Under the HKR quasi-isomorphism the map $B_w$ does indeed become $\wedge dw$: 
locally we have 
\[dw \wedge HKR((1\boxtimes a_1\boxtimes a_2\dots \boxtimes a_q\boxtimes 1)\boxtimes 1)= \frac{1}{q!}dw\wedge da_1\wedge \dots \wedge da_q\]
and 
\begin{footnotesize}
\begin{align*}
HKR(B_{w}(1\boxtimes a_1\boxtimes\dots\boxtimes a_q\boxtimes 1))&=\frac{1}{q+1!}\sum_{i=0}^{k+1}(-1)^{i}da_1\wedge \dots \wedge da_i\wedge dw\wedge da_{i+1}\wedge \dots \wedge da_q\\
&=\frac{1}{q!}dw\wedge da_1\wedge\dots \wedge da_q.\end{align*}
\end{footnotesize}

This gives that the Hochschild homology of $\MF_{loc}(X,\O_X,w)$ is given as the hypercohomology of the complex
\[\dots \stackrel{\wedge dw}{\to} \bigoplus_{i \mbox{ even}} \Omega^i\stackrel{dw\wedge}{\to}
\bigoplus_{i \mbox{ odd}} \Omega^i\stackrel{dw\wedge}{\to}
\bigoplus_{i \mbox{ even}} \Omega^i\stackrel{dw\wedge}{\to}
\bigoplus_{i \mbox{ odd}} \Omega^i\stackrel{dw\wedge}{\to}
\dots\]
where $\ds\bigoplus_{i \mbox{ even}} \Omega^i$ is in even degrees and $dw\wedge$ wedges $dw$ in the first slot. 
\end{proof}

\section{Boundary-bulk and Chern Character}\label{section:atiyah}
Let us recall from \cite{PV:HRR} the definition of the {\em boundary-bulk map}
\[\tau_\E:\Hom_{\MF_{loc}(w)}(\E,\E)\to \mathrm{HH}(\MF_{loc}(X,\L,w)).\]
We have a natural quasi-isomorphism 
\[\Hom_{\MF_{loc}(w)}(\E,\E)\cong Tr(\E\boxtimes \E^{\vee})\]
Then we apply the trace functor the evaluation map $\E\boxtimes \E^{\vee}\to \Delta_\ast S(\O_X)$ to get a map 
\[\tau_\E: \Hom_{\MF_{loc}(X,\L,w)}(\E,\E)\cong Tr(\E\boxtimes \E^{\vee})\to Tr(\Delta_\ast S(\O_X))=\mathrm{HH}(\MF_{loc}(X,\L,w)).\]
Of course, this construction works for any DG-category, we refer the reader to {\em loc. cit.} for the details and also for the proof of the fact that the Chern Character in the sense of \cite{Shkly} is $\tau_\E(id)$. 

Now, having computed the Hochschild homology for the category $\MF_{loc}(X,\O_X,w)$ as $\mathbb R\Gamma(\Omega_{dw})$ and now making the trivial observation that since $\E$ is locally free we have 
\[\RHom(\E,\F)=\mathbb R\Gamma \sHom(\E,\F),\]
one is want to promote the boundary bulk-map to a map in the derived category of sheaves on $X$:
\[\mathcal T_\E:\sHom(\E,\E)\to \Omega_{dw},\]
and thereby understand the particular invariants we wish to compute in two steps, first to get an explicit representative for $\mathcal T_\E$ and then to understand the more classical problem of deducing the induced map on cohomology. 

\begin{lem} \label{lem:sheafifiedTrace} Define a map $\mathcal T_\E:\sHom(\E,\E)\to \Omega_{dw}$ in $D(X)$ by 
\[\sHom(\E,\E)=\E\otimes \E^{\vee}\cong \mathbb L\Delta^\ast(\E\boxtimes \E^{\vee})\stackrel{eval}{\to} \mathbb L\Delta^{\ast}(\Delta_\ast S(\O_X))\cong\Omega_{dw}\]
Then $\tau_\E=\mathbb R\Gamma(\mathcal T_\E)$. 
\end{lem} 
\begin{proof} This is clear.
\end{proof}

We wish now to get a better handle on this map $\mathcal T_\E$. We may resolve a matrix factorization $\E$ by $\epsilon\otimes 1:\Bar_{\tilde w}\tens{S(\O_X)} \E\to \Delta_\ast S(\O_X)\tens{S(\O_X)}\E=\E$. Here we use the short hand $\tens{S(\O_X)}$ between an $\tilde w$ curved module on $X\times X$ and a $w$-curved module on $X$ to mean
\[\Bar_{\tilde w}\tens{S(\O_X)} \E:=(p_1)_\ast(\Bar_{\tilde w}\tens{S(\O_{X^2})} p_2^\ast\E)\]
where $p_1$ and $p_2$ are the natural projections from $X\times X$ to $X$.  Since matrix factorizations are flat, lemma \ref{lem:deltares} implies that this map is a weak equivalence in $Z^0\Qcoh(w)$.
Then the map
\[\E^\vee\otimes (\Bar_{\tilde{w}}\tens{S(\O_X)} \E)\to \E^{\vee}\otimes\E=\sHom(\E,\E)\]
is a quasi-isomorphism of complexes of sheaves.

This gives us an explicit representative for $\mathcal T_\E$ given by the roof
\begin{center}
\begin{tikzpicture}
\matrix(m)[matrix of math nodes, row sep=3em, column sep=3em]
{\E^{\vee}\otimes(\hat{B}_{\tilde w}\tens{S(\O_{X})} \E) &&\E^{\vee}\otimes\E\tens{S(\O_{X})}\Bar_{\tilde w}\\
\E^{\vee}\otimes \E&{}&\Delta^\ast\Bar_{\tilde w}\\
{}&{}&\Omega_{dw}\\};

\path[->, font=\scriptsize]
(m-1-1) edge node[sloped, above=1]{$\sim$} node[auto,swap]{$1\otimes \epsilon\otimes 1$}(m-2-1)
(m-1-1) edge node[auto]{$1\otimes \sigma$} (m-1-3)
(m-1-3) edge node[auto]{$ev\overline{\otimes} 1\otimes 1$} (m-2-3)
(m-2-3) edge node[auto]{$HKR$} (m-3-3);
\end{tikzpicture}
\end{center}

where $ev$ is the evaluation map of $\E^{\vee}$ on $\E$, $\sigma$ is switching the factors in the tensor product and $\overline{\otimes}$ is contraction of tensor. 

The goal now is to construct a natural morphism 
\[ \mathcal Exp(at(\E)): \E\to \Omega_{dw}\tens{S(\O_X)} \E\] in the coderived category of $w$-curved modules, such that 
\begin{small}
\[\mathcal T_\E=str(-\circ  \mathcal Exp(at(\E)):\sHom(\E,\E)\to \sHom(\E,\Omega_{dw}\tens{S(\O_X)} \E)=\sHom(\E,\E)\tens{S(\O_X)}\Omega_{dw}\to \Omega_{dw}\]
\end{small}
This morphism will then be a sort of internal Chern Character for the category of matrix factorization. In what follows we will want to fix $n=dim(X)$. \\

Before we proceed we wish take take a motivational digression and consider the category of complexes of coherent sheaves on $X$. We will follow very closely the treatment from \cite{Markarian}. The idea is that in {\em loc. cit} Markarian constructs an internal Chern Character by exponentiating the Atiyah class map and which takes values in Hochschild homology sheaves. We wish to mimic this construction. The main technical problem, as we will see, is that there is no obvious analog to the Atiyah class for matrix factorizations. But, oddly enough, even though the class $at(\E)$ does not seem to exist, its exponential does.

 We have the exact sequence of $\O_{X^2}$ modules
\[0\to\I/\I^2\to \O_{X^2}/\I^2\to \O_{\Delta}\to 0\]
where $\I$ is the kernel of the multiplication map $\O_{X^2}\to \O_{\Delta}$. We will write $\Omega_{\Delta}$ for $\I/\I^2$ and $\J^1_\Delta$ for $\O_{X^2}/\I^2$. Given an honest complex ($d^2=0$) of sheaves, $\E$, we may ``tensor on the right" by $\E$ to get an exact sequence of $\O_X$-complexes
\begin{equation} \label{eqn:AtiyahE} 0\to \Omega^1\tens{\O_X}\E\to \J^1\otimes_{O_X} \E\to \E\to 0.\end{equation}
where for an $\O_{X^2}$ module $\M$ and an $\O_X$-module $\F$ 
\[\M\tens{\O_X} \F:=(p_1)_\ast(\M\tens{\O_{X^2}} p_2^*(\F))\]
where $p_i:X\times X\to X$ are the standard projections. The extension in \eqref{eqn:AtiyahE} gives an element of \[Ext^1(\E,\Omega^1\tens{\O_X} \E)=Hom_{D(X)}(\E, \Omega^1\tens{\O_X}\E[1]).\] This element, $at(\E):\E\to \Omega\tens{\O_X}\E[1]$, is called the Atiyah class of $\E$.

Composing the morphism $at(\E)$ with itself $i$ times and then wedging forms we obtain a map 
\[\wedge at(E)^i:\E\to \Omega^i\tens{\O_X} \E[i].\] 

Using the isomorphism $\O_{\Delta}\tens{\O_X} \Omega^1\cong \Omega_\Delta$, get a long exact sequence
\[0\to \Omega^{\otimes i-1}\tens{\O_X}\Omega_{\Delta} \to\Omega^{\otimes i-1}\tens{\O_X}\J_\Delta^1\to \dots \to \Omega^1\tens{\O_X}\J^1_\Delta\to \J^1_\Delta\to \O_{\Delta}\]

Tensoring this sequence on the right with $\E$ we get a long exact sequence
\begin{equation} \label{eqn:AtiyahnE} 0\to \Omega^{\otimes i}\tens{\O_X}\E\to \Omega^{\otimes i-1}\tens{\O_X} \J^1(\E)\to \dots \to \Omega^1\tens{\O_X}\J^1(\E)\to \J^1(\E)\to \E.
\end{equation}
Here we denote by $\J^1(\E)$ the tensor product $\J^1_\Delta\tens{\O_X} \E$. One sees easily that this exact sequence represents $\wedge at(E)^i$ as a Yoneda extension and so the map $\wedge at(\E)^i$ is given as the zig-zag
\[\E\leftarrow (\Omega^{\otimes i}\tens{\O_X}\E\to \Omega^{\otimes i-1}\tens{\O_X} \J^1(\E)\to \dots \to \Omega^1\tens{\O_X}\J^1(\E)\to \J^1(\E))\to \Omega^i\otimes \E[i]\] 
where the last map is simply projection onto the last factor followed by wedging forms.

When we try to mimic this construction for curved ($w\neq 0$) modules, the projection onto the last factor is no longer a map in the category we care about. Or more accurately the inclusion of graded $\O_X$ modules $\Omega^i\otimes \E[i]\to \Omega_{dw}\otimes \E[i]$ is not a map of curved modules, unless $i=n$ or $dw=0$. Our first observation is that we can view the exponential of the Atiyah class  as a map from the total complex of the resolution,
\[\Omega^{\otimes n}\tens{\O_X}\E\to \Omega^{\otimes n-1}\tens{\O_X} \J^1(\E)\to \dots \to \Omega^1\tens{\O_X}\J^1(\E)\to \J^1(\E)\]
of $\E$, to $\Omega^\bullet\otimes \E$, by using the various projections onto $\Omega^{\otimes i}\otimes \E$, for $i\leq n$,
where again $n=dim(X)$. The second observation is that we still can in $\MF_{loc}(X,\L,w)$  construct appropriate analogs of this resolution of $\E$. We do this now.

As with the curved bar complex we may use the resolution 
\begin{equation} \label{eqn:AtiyahE2} 0\to \Omega^{\otimes n-1}\tens{\O_X}\Omega_{\Delta}\to \Omega^{\otimes n-1}\tens{\O_X} \J_\Delta^1\to \dots \to \Omega^1\tens{\O_X}\J^1_\Delta \to \J^1_\Delta \to \O_{\Delta}
\end{equation}
to build a $\tilde{w}$ curved complex $\At$ which resolves $\Delta_\ast S(\O_X)$. Set 
\[\mathcal A_i=\begin{cases} \Omega^{\otimes i}\tens{\O_X} \J_\Delta^1 & \mbox{ if $0\leq i<n$}\\
\Omega^{\otimes (n-1)}\tens{\O_X} \Omega_{\Delta}&\mbox{ if $i=n$}\\
0& \mbox{else} \end{cases} \]
Then define the graded components of $\At$ by folding 2-periodically:
\[\At_i=\bigoplus_{j\equiv i\mod 2} \mathcal A_j.\]
 We have the differential, $m:\At_i\to \At_{i+1}$ coming from the resolution \ref{eqn:AtiyahE2} which (locally) is given by the equation 
\begin{align*} m(da_1\otimes da_2\otimes \dots\otimes da_n\otimes a_0\boxtimes a_{n+1})=a_0&a_{n+1}da_1\otimes \dots da_{n-1}\otimes a_{n}\boxtimes 1\\
&-a_0a_{n+1}da_1\otimes \dots da_{n-1}\otimes 1\boxtimes a_{n-1}.
\end{align*}
Here we have chosen indices in preparation for certain morphisms involving the curved bar complex. Again we use $\boxtimes$ to emphasize external tensor. Depending on our purposes, i.e. whether we want to emphasize or deemphasize the role of $J^1_\Delta$ in the tensor $\Omega^{\otimes q}\otimes J^1_\Delta$,  we will alternatively simply write
\[a_0da_1\otimes\dots \otimes da_{q}\boxtimes a_{q+1}=a_1\otimes \dots \otimes a_q\otimes a_0\boxtimes a_{q+1}\] 
 Coordinate free, this map $m$ is simply induced by the multiplication map $\J^1_{\Delta}\to \O_{\Delta}$ followed by the isomorphism $\Omega^{\otimes i}\tens{\O_X} \O_{\Delta}\cong \Omega^{\otimes (i-1)} \tens{\O_X} \Omega_{\Delta}$ and then the inclusion \[\Omega^{\otimes(i-1)}\otimes \Omega_{\Delta}\to \Omega^{\otimes(i-1)} \otimes \J^1.\] And, of course, on the summand 
$\mathcal A_n=\Omega^{\otimes (n-1)}\tens{\O_X} \Omega_{\Delta}$,
$m$ is simply the inclusion of $\Omega^{\otimes(n-1)}\tens{\O_X}\Omega_{\Delta}$ into $\Omega^{\otimes(n-1)}\tens{\O_X}\J_\Delta^1$. 
To curve $\At$ by $\tilde w$ we add a second differential $B_{dw}$ given by the formula

\[B_{dw}(\omega_1\otimes \dots \otimes \omega_n\otimes a_0\boxtimes a_{q+1})=\sum_{i=0}^q(-1)^i\omega_1\otimes \dots \otimes \omega_i\otimes dw\otimes \omega_{i+1}\otimes \dots \otimes\omega_q\otimes a_0\boxtimes a_{q+1}.\]

As with $\Bar_{\tilde{w}}$, we may picture $\At$ as the total complex of the bicomplex:
 \begin{center}
 \begin{tikzpicture}[
cross line/.style={preaction={draw=white, -,
line width=6pt}},
descr/.style={fill=white,inner sep=2.5pt}]

 \matrix(m)[matrix of math nodes, row sep=2em, column sep=3em]
 {{}&\vdots&\vdots&\vdots\\
 0&\dots&0&0&0&\\
 \Omega^{\otimes n}\otimes \J^1_{\Delta}&\dots&\Omega^{\otimes 2}\otimes \J^1_{\Delta}&\Omega\otimes \J^1_{\Delta}&\J^1_{\Delta}\\
 0&\dots&0&0&0\\
 \Omega^{\otimes n}\otimes \J^1_{\Delta}&\dots&\Omega^{\otimes 2}\otimes \J^1_{\Delta}&\Omega\otimes \J^1_{\Delta}&\J^1_{\Delta}\\
0&\dots&0&0&0\\
&\vdots&\vdots&\vdots\\};

\path[->]
(m-2-1) edge[densely dotted] node[auto]{} (m-2-2)
(m-2-2) edge[densely dotted] node[auto]{} (m-2-3)
(m-2-3) edge[densely dotted] node[auto]{} (m-2-4)
(m-2-4) edge[densely dotted] node[auto]{} (m-2-5)
(m-3-1) edge node[auto]{$b$} (m-3-2)
(m-3-2) edge node[auto]{$b$} (m-3-3)
(m-3-3) edge node[auto]{$b$} (m-3-4)
(m-3-4) edge node[auto]{$b$} (m-3-5)
(m-4-1) edge[densely dotted] node[auto]{} (m-4-2)
(m-4-2) edge[densely dotted] node[auto]{} (m-4-3)
(m-4-3) edge[densely dotted] node[auto]{} (m-4-4)
(m-4-4) edge[densely dotted] node[auto]{} (m-4-5)
(m-5-1) edge node[auto]{$b$} (m-5-2)
(m-5-2) edge node[auto]{$b$} (m-5-3)
(m-5-3) edge node[auto]{$b$} (m-5-4)
(m-5-4) edge node[auto]{$b$} (m-5-5)
(m-6-1) edge[densely dotted] node[auto]{} (m-6-2)
(m-6-2) edge[densely dotted] node[auto]{} (m-6-3)
(m-6-3) edge[densely dotted] node[auto]{} (m-6-4)
(m-6-4) edge[densely dotted] node[auto]{} (m-6-5)
(m-3-5) edge[cross line] node[left=16pt,above,sloped]{$B_{dw}$} (m-1-4)
(m-5-5) edge[cross line] node[left=16pt,above,sloped]{$B_{dw}$} (m-3-4)
(m-7-5) edge[cross line] node[left=16pt,above,sloped]{$B_{dw}$} (m-5-4)
(m-3-3) edge[cross line] node[left=16pt,above,sloped]{$B_{dw}$} (m-1-2)
(m-3-4) edge[cross line] node[left=16pt,above,sloped]{$B_{dw}$} (m-1-3)
(m-3-2) edge[cross line] node[left=16pt,above,sloped]{$B_{dw}$} (m-1-1)
(m-7-3) edge[cross line] node[left=16pt,above,sloped]{$B_{dw}$} (m-5-2)
(m-7-4) edge[cross line] node[left=16pt,above,sloped]{$B_{dw}$} (m-5-3)
(m-7-2) edge[cross line] node[left=16pt,above,sloped]{$B_{dw}$} (m-5-1)
(m-5-3) edge[cross line] node[left=16pt,above,sloped]{$B_{dw}$} (m-3-2)
(m-5-4) edge[cross line] node[left=16pt,above,sloped]{$B_{dw}$} (m-3-3)
(m-5-2) edge[cross line] node[left=16pt,above,sloped]{$B_{dw}$} (m-3-1);
\end{tikzpicture}
\end{center}

Now we claim that $\At$ imbued with the differential $B_{dw}-m\gamma$ is a $\tilde w$ curved module, where $\gamma$ is the grading operator with respect to forms, i.e. $\gamma|_{\Omega^q\otimes J^1_{\Delta}}=(-1)^q$. Indeed the computations
\begin{footnotesize}
\begin{align*}
m\gamma B_{dw}(a_0da_1\otimes\dots &\otimes da_q\boxtimes a_{q+1})\\
&=(-1)^{q+1}\left[\sum_{i=0}^q(-1)^im(a_0da_1\otimes\dots\otimes dw\otimes \dots \otimes da_q \boxtimes a_{q+1})\right]\\
&=(-1)^{q+1}\Bigg[(-1)^q(a_0a_{q+1}da_1\otimes \dots \otimes da_q\otimes w\boxtimes 1-a_0a_{q+1}da_1\otimes \dots da_q\otimes 1\boxtimes w)\\
&~~+\sum_{i=0}^{q-1}(-1)^ia_0a_{q+1}da_1\otimes \dots \otimes da_i\otimes dw\otimes da_{i+1}\otimes \dots \otimes da_{q-1}\otimes a_q\boxtimes 1\\
&~~-\sum_{i=0}^{q-1}(-1)^ia_0a_{q+1}da_1\otimes \dots \otimes da_i\otimes dw\otimes da_{i+1}\otimes \dots \otimes da_{q-1}\otimes 1\boxtimes a_q\Bigg]
\end{align*}
\end{footnotesize}
and 
\begin{small}
\begin{align*}
B_{dw} &m\gamma(da_1\otimes\dots \otimes da_q\otimes a_0\boxtimes a_{q+1})\\
&=(-1)^q\Bigg[\sum_{i=0}^{q-1}(-1)^ia_0a_{q+1}da_1\otimes \dots \otimes da_i\otimes dw\otimes da_{i+1}\otimes \dots \otimes da_{q-1}\otimes a_q\boxtimes 1\\
&\quad-\sum_{i=0}^{q-1}(-1)^ia_0a_{q+1}da_1\otimes \dots \otimes da_i\otimes dw\otimes da_{i+1}\otimes \dots \otimes da_{q-1}\otimes 1\boxtimes a_q\Bigg]
\end{align*}
\end{small}
show that 

\begin{align*}
(B_{dw}-m\gamma)^2&=-B_{dw}m\gamma-m\gamma B_{dw}(da_1\otimes \dots \otimes da_q\otimes a_0\boxtimes a_{q+1})\\
&=a_0a_{q+1}da_1\otimes \dots \otimes da_{q}\otimes (w\boxtimes 1- 1\boxtimes w)\\
&=da_1\otimes \dots \otimes da_{q}\otimes (a_0a_{q+1}w\boxtimes 1- a_0a_{q+1}\boxtimes w)
\end{align*}

The final observations are that $\tilde{w}$ acts on $\Omega^{\otimes q}\tens{\O_X} \J_\Delta^1$ by 
\[\tilde{w}\cdot \omega_1\otimes \dots \omega_q\otimes a_0\boxtimes a_{q+1}=w\omega_1\otimes \dots \omega_q\otimes a_0\boxtimes a_{q+1}w=\omega_1\otimes \dots \omega_q\otimes wa_0\boxtimes a_{q+1}w\]
and the difference between this action and the above computation for $(B_{dx}-m\gamma)^2$  is
\[a_0a_{q+1}w\boxtimes 1-a_0a_{q+1}\boxtimes w-wa_0\boxtimes a_q+a_0\boxtimes wq=(a_0\boxtimes 1)(w\boxtimes 1-1\boxtimes w)(a_{q+1}\boxtimes 1-1\boxtimes a_{q+1})\]
which is 0 in $\J^1_\Delta$. Therefore the map $(B_{dw}-m\gamma)^2=-B_{dw}m\gamma-m\gamma B_{dw}$ is indeed multiplication by $\tilde w$. 

Now there are maps $\pi:\O_{\mathfrak X^{q+2}}\to \Omega^{\otimes q}\tens{\O_X} \J_\Delta^1$ given by 
\[\pi(a_0\boxtimes a_1\boxtimes \dots \boxtimes a_q\boxtimes a_{q+1})=a_0da_1\otimes \dots \otimes da_q\boxtimes a_{q+1}\]
It is easy to see $\pi B_{w}=B_{dw}\pi$. We observe that for an elements of $\O_{\mathfrak X^{q+1}}$ of the form $a_0\boxtimes \dots \boxtimes a_ia_{i+1}\boxtimes \dots \boxtimes a_{q+1}$, with $0<i<q+1$ we have 
\begin{align*}
\pi(a_0\boxtimes \dots \boxtimes a_ia_{i+1}\boxtimes \dots \boxtimes a_{q+1})&=da_1\otimes \dots \otimes d(a_ia_{i+1})\otimes da_q\otimes a_0\boxtimes a_{q+1}\\
&= a_0a_{i+1}da_1\otimes \dots \otimes da_i\otimes da_{i+2}\otimes\dots \otimes da_q\boxtimes a_{q+1}\\
&\quad +a_0a_{i}da_1\otimes \dots \otimes da_{i-1}\otimes da_{i+1}\otimes\dots \otimes da_q\boxtimes a_{q+1}\\
&=a_0a_ida_1\otimes \dots \otimes \widehat{da_i}\otimes \dots \otimes da_q\boxtimes a_{q+1}\\
&\quad +a_0a_ida_1\otimes \dots \otimes \widehat{da_{i+1}}\otimes \dots \otimes da_q\boxtimes a_{q+1}\\
\end{align*}
where $\widehat{da_i}$ indicates to omit this tensor. We also have 
\begin{align*}\pi(a_0a_1\boxtimes a_2\boxtimes \dots \boxtimes a_{q+1})&=a_0a_1\otimes a_2\otimes\dots \otimes da_q\boxtimes a_{q+1}\\
&=a_0a_1\otimes \widehat{da_1}\otimes da_2\otimes \dots \otimes da_q\boxtimes a_{q+1}\\
\end{align*}
So, again using $b$ for the Hochschild differential, we have
\begin{small}
\begin{align*} 
\pi b(a_0\boxtimes \dots \boxtimes a_{q+1})&=\sum_{i=0}^{q}(-1)^i\pi(a_0\boxtimes \dots \boxtimes a_ia_{i+1}\boxtimes \dots \boxtimes a_{q+1})\\
&=(-1)^qa_0 da_1\otimes\dots\otimes da_{q-1}\boxtimes a_qa_{q+1}+a_0a_1da_2\otimes \dots\otimes da_n \boxtimes a_{n+1}\\
 &\qquad+\sum_{i=1}^{q-1}(-1)^{i}a_0a_i\otimes da_1\otimes\dots \otimes\widehat{da_i}\otimes \dots\otimes  da_q\boxtimes da_{q+1} \\
 &+\sum_{i=1}^{q-1}(-1)^{i}a_0a_{i+1}\otimes da_1\otimes\dots\otimes \widehat{da_{i+1}}\otimes \dots \otimes da_q\boxtimes da_{q+1}\\
&=(-1)^qa_0 da_1\otimes\dots\otimes da_{q-1}\boxtimes a_qa_{q+1}+a_0a_1da_2\otimes \dots\otimes da_n \boxtimes a_{n+1}\\
 &\qquad+\sum_{i=1}^{q-1}(-1)^{i}a_0a_i\otimes da_1\otimes\dots \otimes\widehat{da_i}\otimes \dots\otimes  da_q\boxtimes da_{q+1} \\
 &\qquad+\sum_{i=2}^{q}(-1)^{i-1}a_0a_{i}\otimes da_1\otimes\dots\otimes \widehat{da_{i}}\otimes \dots \otimes da_q\boxtimes da_{q+1}\\
 &=(-1)^qa_0 da_1\otimes\dots\otimes da_{q-1}\boxtimes a_qa_{q+1}+a_0a_1da_2\otimes \dots\otimes da_q \boxtimes a_{q+1}\\
 &\qquad-a_0a_1da_2\otimes\dots \otimes\widehat{da_i}\otimes \dots\otimes  da_q\boxtimes da_{q+1}\\
 &\qquad+(-1)^{q-1}a_0da_q da_1\otimes \dots \otimes da_{q-1}\boxtimes a_{q+1} \\
&=(-1)^{q+1}m \pi(a_0\boxtimes \dots \boxtimes a_{q+1})
\end{align*}
\end{small}
The above discussion proves the following lemma:
\begin{lem} \label{lem:pi} The map $\pi: \Bar_{\tilde w}\to \At$ is a closed morphism of $\tilde w$-curved modules.
\end{lem}

 Incidentally this discussion also explains the appearance of the grading operator in the horizontal direction. 

\begin{rmk} \label{rmk:piWeakEquivalence} It is clear that $\pi: \Bar_{\tilde w}\to \At$ is a weak equivalence of $\tilde w$ curved modules on $X\times X$, since both $\Bar_{\tilde w}$ and $\At$ are weakly equivalent to $\Delta_\ast S(\O_X)$ via projection.
\end{rmk}

As a piece of notation, for $\E\in \Qcoh(X,\O_X,w)$, we define 
\[\At(\E):=\At\tens{S(\O_{X})}\E:=(p_1)_\ast(\At\tens{S(\O_{X^2})}p_2^\ast\E).\]

\begin{lem}\label{lem:wedgeq} Let $\wedge: \Omega^{\otimes q}\otimes \J^1(\E)\to \Omega^q\otimes \E$ denote the anti-symmetrization map:
\[\wedge(a_0da_1\otimes\dots \otimes da_q\boxtimes e)=a_0da_1\wedge\dots \wedge da_q\otimes e\]
Then the map \[\sum_{i=0}^n \frac{\wedge}{i!}: \At(\E)\to \Omega_{dw}\tens{S(\O_X)} \E\] 
gives a closed degree 0 morphism of $w$-curved modules. 
\end{lem}
\begin{proof} 
This follows from the calculations
\begin{footnotesize}
\begin{align*} 
\frac{\wedge}{(q+1)!}B_{dw}(a_0da_1\otimes \dots & \otimes da_q\boxtimes e)\\
&= \frac{1}{(q+1)!}\sum_{i=0}^q(-1)^ia_0da_1\wedge \dots \wedge da_i\wedge dw\wedge da_{i+1}\wedge \dots \wedge da_{q}\otimes e\\
&=\frac{1}{(q+1)!}\sum_{i=0}^qa_0dw\wedge da_1\wedge \dots \wedge da_q\otimes e\\
&=\frac{1}{q!}dw\wedge a_0da_1\wedge \dots \wedge da_q\otimes e\\
&=dw\wedge \left(\frac{\wedge}{q!}(a_0a_1\otimes \dots \otimes da_n\boxtimes e)\right)\\
\end{align*} 
\end{footnotesize}
and 
\begin{align*} 
\frac{\wedge}{(q-1)!}m(a_0da_1\otimes \dots \otimes da_q\boxtimes e)&= \frac{1}{(q-1)!}(a_0a_qda_1\wedge \dots \wedge da_{q-1}\otimes e)\\
&\quad -\frac{1}{(q-1)!}(a_0da_1\wedge \dots \wedge da_{q-1}\otimes a_qe)\\
&=0
\end{align*}
 and the observation that the differential on $\E$ obviously commutes with the map $\sum_i \wedge_i$. 
\end{proof} 

\begin{dfn} \label{dfn:expAt} Define the map $\mathcal Exp(at(\E)):\E\to \Omega_{dw}\otimes \E$ in the category $D^{co}\Qcoh(X,\O_X,w)$ by the roof \begin{center}
\begin{tikzpicture}
\matrix(m)[matrix of math nodes, row sep=2em, column sep=2em]{{}&\At(\E)&{}\\ \E&{}&\Omega_{dw}\otimes \E\\};
\path[->]
(m-1-2) edge node[sloped,below=1]{$\sim$} node[sloped, above=1]{$\pi$} (m-2-1)
(m-1-2) edge node[auto]{$\Sigma \frac{\wedge}{i!}$}(m-2-3);
\end{tikzpicture} 
\end{center}
\end{dfn}

\begin{lem}\label{lem:sheafTrace} The sheafified boundary bulk map $\mathcal T_\E:\sHom_{\S(\O_X)}(\E,\E)\to \Omega_{dw}$ is given by $str(-\circ \mathcal Exp(at(\E)))$, where $str:\sHom_{S(\O_X)}(\E,\E)\to S(\O_X)$ is the super-trace map. 
\end{lem}
\begin{proof} Recall from the discussion at the beginning of this section that we have the following representative for $\mathcal T_\E$  
\begin{center}
\begin{tikzpicture}
\matrix(m)[matrix of math nodes, row sep=3em, column sep=3em]
{\E^{\vee}\otimes(\hat{B}_{\tilde w}\tens{S(\O_{X})} \E) &&\E^{\vee}\otimes\E\tens{S(\O_{X})}\Bar_{\tilde w}\\
\E^{\vee}\otimes \E&{}&\Delta^\ast\Bar_{\tilde w}\\
{}&{}&\Omega_{dw}\\};

\path[->, font=\scriptsize]
(m-1-1) edge node[sloped,above=.5]{$\sim$} node[auto,swap]{$1\otimes \rho\otimes 1$}(m-2-1)
(m-1-1) edge node[auto]{$1\otimes \sigma$} (m-1-3)
(m-1-3) edge node[auto]{$ev\overline{\otimes} 1\otimes 1$} (m-2-3)
(m-2-3) edge node[auto]{$HKR$} (m-3-3);
\end{tikzpicture}
\end{center}
From lemmas \ref{lem:pi} and \ref{lem:wedgeq} we may complete this diagram to the following picture
\begin{center}
\begin{tikzpicture}
\matrix(m)[matrix of math nodes, row sep=3em, column sep=1.5em]
{\E^{\vee}\otimes(\hat{B}_{\tilde w}\tens{S(\O_{X})} \E) &&\E^{\vee}\otimes\E\tens{S(\O_{X})}\Bar_{\tilde w}\\
\E^{\vee}\otimes \E&\E^{\vee}\otimes \At(\E) &\Delta^\ast\Bar_w \\
{}&\E^{\vee}\otimes (\Omega_{dw} \otimes \E)&\Omega_{dw}\\
\sHom(\E,\E)&\sHom(\E,\Omega_{dw}\otimes \E)&\Omega_{dw}\\};

\path[->, font=\scriptsize]
(m-1-1) edge node[sloped]{$\sim$} node[auto,swap]{$1\boxtimes \rho\otimes 1$}(m-2-1)
(m-1-1) edge node[auto]{$1\boxtimes \sigma\otimes 1$} (m-1-3)
(m-2-2) edge node[auto]{$\sim$} (m-2-1)
(m-1-1) edge node[auto]{$1\boxtimes \pi \otimes 1$} (m-2-2)
(m-2-2) edge node[auto]{} (m-3-2)
(m-3-2) edge node[auto]{} (m-3-3)
(m-2-1) edge node[auto]{} (m-4-1)
(m-3-2) edge node[auto]{} (m-4-2)
(m-3-3) edge node[auto]{} (m-4-3)
(m-4-2) edge node[auto]{str} (m-4-3)
(m-1-3) edge node[auto]{$ev\overline{\otimes} 1\otimes 1$} (m-2-3)
(m-2-3) edge node[auto]{$HKR$} (m-3-3);
\end{tikzpicture}
\end{center}

It is easy to check that everything commutes, at least perhaps with the added observation that the bottom arrow, which we have by abuse simply called $str$, first commutes the tensors under the isomorphism \[\sHom(\E,\Omega_{dw}\otimes \E)\cong \sHom(\E,\E)\otimes \Omega_{dw}\] 
before applying the super-trace. 

\end{proof}

\label{rmk:sheaves}Taking $w=0$, our results give us directly information about $\mathbb Z_2$ complexes of vector bundles on $X$. \footnote{There is a mild issue here with respect to \ref{thm:cptgen} which requires that $w$ not be a zero divisor, however the conclusion of this theorem is well-know to still hold when $w=0$, so there is no problem.} Moreover, one checks (essentially by taking a standard tensor product on complexes, rather than the folded tensor we use for matrix factorizations) that all of the above constructions and theorem go through. This gives the following result, which follows formally from \cite{Cald:MukaiII} and \cite{Ramadoss}. However, there seems to be a problem in Ramadoss's proof in \cite{Ramadoss}: in the proof of Proposition 2 he uses without explaining the coincidence of the two versions of the Chern character of $\O_{\Delta}$, one defined in \cite{Cald:MukaiII} and the one one coming from DG theory.

\begin{thm} \label{thm:coincide}The DG Chern Character map for perfect complexes on smooth $X$ in the sense of \cite{Shkly} coincides with the classical Chern Character.
\end{thm}
\begin{proof} We apply our remark \ref{rmk:sheaves} to lemma \ref{lem:sheafTrace} and \ref{lem:sheafifiedTrace} to find that the DG Chern Character of a bounded complex of locally free sheaves $\E$ is given by 
\[\mathbb R\Gamma(str(\mathcal Exp(at(E))))\in  \mathbb R\Gamma(\bigoplus \Omega^{i})=HH(X).\] This according to \cite{Markarian} is exactly the classical Chern Character.
\end{proof}

\section{A Formula for the Boundary-bulk Map} \label{section:formula}

In this section we wish to develop a global analog of the Chern character formula for global matrix factorizations computed for a formal disk in \cite{PV:HRR}. There should be some question about what such an analog could be since globality generally prohibits formulas, at least formulas involving coordinates. Another option would be to relate to Chern character to certain classes which exist globally, e.g. Chern classes or the Atiyah class. At some level we have already done this and at another we have already discussed the obstruction to doing so. We should probably also point out here that we do not know what Chern classes are for matrix factorizations.

We have taken the task of finding a global Chern character formula and more generally the boundary bulk map to mean the following: understand the image of the boundary bulk map in some computable model for $\mathbb R\Gamma(\Omega_{dw})$. This will be a Cech model and we will give our formula in terms of local connections on a Cech cover.

\begin{lem} \label{lem:connection}Let $\E$ be a matrix factorization with curved differential $e$. Suppose $\nabla$ is a connection on $\E$, i.e. $\nabla$ consists of standard connections on underlying graded components $\nabla_i:\E_0\to\Omega\otimes \E_i$ for $i=0,1$. Then the morphism
\begin{center}
\begin{tikzpicture}
\matrix(m)[matrix of math nodes, row sep=2em, column sep=2em]{{}&\At(\E)&{}\\ \E&{}&\Omega_{dw}\otimes \E\\};
\path[->]
(m-1-2) edge node[auto,sloped]{$\sim$} (m-2-1)
(m-1-2) edge (m-2-3);
\end{tikzpicture} 
\end{center}
which represents the map $\E xp(at(\E))$ in the coderived category from definition \ref{dfn:expAt} is given by the map of $w$-curved complexes 
\[exp(at(\E))=\sum_{i=0}^n \frac{\wedge[\nabla,e]^i}{i!}.\]
\end{lem}

\begin{proof} Given a connection $\nabla$ on $\E$ we obtain splittings $\J^1(\E_i)=\Omega^1\otimes \E_i\oplus \E_i$ under these splittings the induced map by $e$ on $\J^1(\E)$ becomes 
\[\left(\begin{matrix} 1\otimes e & [\nabla,e]\\
0&e\end{matrix}\right),\]
The map \[m:\Omega^{\otimes q+1}\otimes \E_i\oplus \Omega^{q}\otimes \E_i\to \Omega^{q}\otimes \E_i\oplus \Omega^{q-1}\E_i\]
is simply given by the projection
\[\left( \begin{matrix} 0&1\\0&0\end{matrix}\right),\]
and the map $B_{dw}$ splits as 
\[\left(\begin{matrix} B'_{dw}&0\\0&B''_{dw}\end{matrix}\right)\]
where 
$B'_{dw}:\Omega^{\otimes q+1}\otimes \E_i\mapsto \Omega^{\otimes(q+2)}\otimes \E_i$
is given by 
\[B'_{dw}(a_0da_1\otimes \dots \otimes da_{q+1}\otimes e)=\sum_{j=0}^q(-1)^ja_0da_1\otimes \dots\otimes da_j\otimes dw\otimes da_{j+1}\otimes \dots \otimes da_{q+1}\otimes e\]
and
$B''_{dw}:\Omega^{\otimes q}\otimes \E_i\mapsto \Omega^{\otimes(q+1)}\otimes \E_i$
is given by
\[B''_{dw}(a_0da_1\otimes \dots \otimes a_q\otimes e)= \sum_{j=0}^q(-1)^ja_0da_1\otimes \dots\otimes da_j\otimes dw\otimes da_{j+1}\otimes \dots  \otimes da_{q}\otimes e.\]

Using these splittings we may view $\sum_{i=0}^n [\nabla,e]^i$ as a degree 0 map $\E\to \At(\E)$ then it will suffice to show that that map is a closed morphism of $w$-curved modules since it obviously splits the weak-equivalence $\At(\E)\to \E$ induced by projecting and we also have 
\[\left(\sum_{i}\frac{\wedge}{i!}\right) \circ\left(\sum_i [\nabla,e]^i\right)=exp(at(\E)).\] 

We first need to make the simple calculation:
\[e[\nabla,e]+[\nabla,e]e=\nabla w-w\nabla=dw.\]

Now 
\begin{small}
\begin{eqnarray*}
e\gamma[\nabla,e]^q-[\nabla,e]^qe &=&(-1)^{q} e[\nabla,e]^q-[\nabla,e]^qe\\
&=&(-1)^{q}\sum_{i=0}^{q-1}(-1)^{i}[\nabla,e]^{i}(e[\nabla,e]+[\nabla,e]e)[\nabla, e]^{q-i-1}\\
&=&-\sum_{i=0}^{q-1}(-1)^{q-i-1}[\nabla,e]^{i}\otimes dw\otimes [\nabla,e]^{q-i-1}\\
&=&-B''_{dw}[\nabla,e]^{q-1}\\
\end{eqnarray*}
\end{small}
The equality going from line 3 to 4 above is a bit tricky: The map $dw\otimes-$ sends 
\[a_0da_1\otimes da_{q-i-1}\otimes e\mapsto a_0da_1\otimes da_{q-i-1}\otimes dw\otimes e\] and then for the composition we have
\begin{footnotesize}
\begin{center}
\begin{tikzpicture}
\matrix(m)[matrix of math nodes, row sep=2em, column sep =3em]{\E_j&\Omega^{\otimes q-i-1}\otimes \E_{j+q-i-1}&{}\\
{}& \Omega^{\otimes q-i-1}\otimes \Omega^1\otimes \E_{j+q-i-1}&\Omega^{\otimes q-i-1}\otimes \Omega^1\otimes \Omega^i\otimes \E_{j+q-1}\\};

\path[->, font=\footnotesize]
(m-1-1) edge node[auto]{$[d,\nabla]^{q-i-1}$} (m-1-2)
(m-1-2) edge node[auto]{$dw\otimes-$} (m-2-2)
(m-2-2) edge node[auto]{$[d,\nabla]^{i}$} (m-2-3);
\end{tikzpicture}
\end{center} 
\end{footnotesize}
so computing $[\nabla,e]^i\otimes dw\otimes [e,\nabla]^{q-i-1}$ is the same as computing $[\nabla,e]^{q-1}$ and then inserting $dw$ in the $q-i-1$st slot. 
Now we may compute
\begin{footnotesize}
\begin{align*}\left[
\begin{pmatrix}
B'_{dw} & 0\\
0       & B''_{dw}
\end{pmatrix}
\right.&\left.-\begin{pmatrix} 
0 & \gamma\\
0 & 0
\end{pmatrix}
+\begin{pmatrix}
d_{\E}\gamma & [\nabla,e]\gamma\\
0            & e\gamma 
\end{pmatrix}\right]\begin{pmatrix} 
0\\ \sum_{i=0}^n[\nabla,e]^i\end{pmatrix}\\
&=\sum_{i=0}^n\begin{pmatrix}
0 \\ B''_{dw}[\nabla,e]^i\end{pmatrix}
-\sum_{i=1}^n\begin{pmatrix}
(-1)^i[\nabla, e]^i\\0
\end{pmatrix}\\
&\qquad+\sum_{i=0}^n\begin{pmatrix} 0\\(-1)^ie[\nabla,e]^i\end{pmatrix}+\sum_{i=0}^n\begin{pmatrix} (-1)^i[\nabla,e]^{i+1}\\0\end{pmatrix}\\
&=\begin{pmatrix} 0\\
\sum_{i=0}^n[\nabla,e]^ie
\end{pmatrix}
\end{align*}
\end{footnotesize}
which finishes the proof. Note that the second sum on the second line starts at $i=1$ because the component of the differential on $\At(E)$ coming from $m$ is 0 on $\J^1(\E)$. 
\end{proof} 

\begin{rmk} In the case when $X$ is a formal disk, we recover the formula for the Chern Character from \cite{PV:HRR} by identifying the cohomology $\mathbb R\Gamma(\Omega_X)$ with the Tyurina algebra .
\end{rmk}

Naturality of the Chern character (or more generally the boundary bulk map) implies that it commutes with restriction to open subschemes. The above lemma tells us what happens to the Chern Character upon restriction to open affine subschemes. Of course, upon restricting we loose information. The following lemmas describe how we can go the other direction. That is, they give us a method to take  this local data (along with an appropriate collection of homotopies) to a global a global morphism to Cech Cohomology. 

\begin{lem} \label{lem:gluing} Let $(\F,d_\F)$ and $(\G, d_\G)$ be complexes of sheaves (or $w$-curved $S(\O_X)$ modules) on $X$ and $U_1,\dots, U_n$ be a Cech cover of $X$. Denote $\G_{i_0\dots i_p}=(j_{i_0\dots j_p})_\ast\G|_{U_{i_0}\cap\dots \cap U_{i_p}}$, where $j_{p_0\dots p_n}$ is the inclusion of $U_{i_0}\cap\dots \cap U_{i_p}$ into $X$. Suppose we are given the following data: for each $0\leq p\leq n$ and each tuple $i_0i_1\dots i_p$ with $1\leq i_0<i_1<\dots<i_p\leq n$ we have a map 
\[f_{i_0\dots i_p}: \F\to \G_{i_0\dots i_p}[p]\]
such that 
\[d_\G f_{i_0\dots i_p}-(-1)^pf_{i_0\dots i_p}d_\F=\sum_{j=0}^p(-1)^kf_{i_0\dots \widehat{i_k}\dots i_p}|_{U_{i_0\dots i_p}},\]
then the map $f: \F\to Cech(\G)$ defined on $\F^q$ by
\[f=\sum_{p}\sum_{i_0\dots i_p}(-1)^{\frac{p(p-1)}{2}}f_{i_0\dots i_p}\]
is a closed degree 0 map of complexes.
\end{lem}

\begin{proof}
First observe that $f_{i_0 \dots i_p}$ takes $\F^q$ to $\G^{q-p}_{i_0 \dots i_p}$ and $\G^{q-p}_{i_0 \dots i_p}$ lives in degree $q$ of $Cech(\G)$, therefore the map $f$ is indeed degree 0.

If we consider the composition $cf$ where $c$ is the Cech differential, we get \begin{align*}
cf&=\sum_{p=0}^n\sum_{i_0\dots i_p}(-1)^{\frac{p(p-1)}{2}}cf_{i_0\dots i_p}\\
&=\sum_{p=0}^n\mathop{\sum_{\mathbf{i}=i_0\dots i_p}}_{\mathbf{j}=j_0\dots j_{p+1}}(-1)^{\frac{p(p-1)}{2}}\sigma(\mathbf{i},\mathbf{j})f_{i_0\dots i_p}|_{U_{j_0\dots j_{p+1}}}
\end{align*}

where
\[\sigma(\mathbf{i},\mathbf{j})=\begin{cases} 1&\mathrm{if}~i_0\dots i_p=j_0\dots \hat{j_k}\dots j_{p+1},~k~\mathrm{even}\\
-1&\mathrm{if}~i_0\dots i_p=j_0\dots \hat{j_k}\dots j_{p+1},~k~\mathrm{odd}\\
0&\mathrm{else}
\end{cases}\]

On the other hand,by our assumption on $f_{i_0\dots i_p}$, we have 
\begin{align*}
\gamma d_\G f-f d_\F&=
\sum_{p=0}^n\sum_{i_0\dots i_p}(-1)^{\frac{p(p-1)}{2}+p}(d_\G f_{i_0\dots i_p}-(-1)^pf_{i_0\dots i_p}d_\F)\\
&=\sum_{p=0}^n\sum_{i_0\dots i_p}(-1)^{\frac{p(p-1)}{2}+p}\sum_{k=0}^p(-1)^kf_{i_0\dots \widehat{i_k}\dots i_p}|_{U_{i_0\dots i_p}}\\
&=\sum_{p=0}^n\mathop{\sum_{\mathbf{i}=i_0\dots i_p}}_{\mathbf{j}=j_0\dots j_{p+1}}(-1)^{\frac{(p+1)(p)}{2}+p+1}\sigma(\mathbf{i},\mathbf{j})f_{i_0\dots i_p}|_{U_{j_0\dots j_{p+1}}}\\
&=-\sum_{p=0}^n\mathop{\sum_{\mathbf{i}=i_0\dots i_p}}_{\mathbf{j}=j_0\dots j_{p+1}}(-1)^{\frac{p(p-1)}{2}}\sigma(\mathbf{i},\mathbf{j})f_{i_0\dots i_p}|_{U_{j_0\dots j_{p+1}}}\\
&=-cf\\
\end{align*}
where $\gamma$ is the grading operator on the Cech complex: $\gamma|_{\G_{i_0\dots i_p}}=(-1)^{p}$. 
\end{proof}

\begin{lem} \label{lem:maps}
\begin{small}
Let $\E$ be a matrix factorization with curved differential $e$. Let  $\nabla_{j}$ be a choice of connection on $U_{j}$, where $\{U_j\}|_{j=1}^N$ is a Cech cover of $X$. The collection of maps
\[f_{i_0\dots i_p}:\E\to \At(\E)_{i_0\dots i_p}[p],\]
\[f_{i_0\dots i_p}=\sigma_{i_0}\sum_{k_0,k_1\dots, k_p} \tau_p(k_0,\dots,k_p)[e,\nabla_{i_0}]^{k_0}(\nabla_{i_0}-\nabla_{i_1})[e,\nabla_1]^{k_1}(\nabla_{i_1}-\nabla_{i_2})\dots 
[e,\nabla_{i_p}]^{k_p} \] $\tau_p(k_0,\dots k_p)=(-1)^{\sum_{j=0}^pj(k_j+1)}$ and $\sigma_{i_0}:\Omega^{\otimes q}\otimes \E\to \Omega^{\otimes q}\otimes \J^1(\E)$ is the splitting induced by $\nabla_{i_0}$
\end{small}
satisfies the hypothesis of lemma \ref{lem:gluing}.
\end{lem}

\begin{proof} 

We know from the proof of lemma \ref{lem:connection} that \[(-1)^k e [\nabla_{i_0},e]-[\nabla_{i_0},e]^ke=-B_{dw} [\nabla_{i_0},e]^{k-1}\] so that \[\gamma e f_{i_0}-f_{i_0}e=-B_{dw} f_{i_0} .\] We claim that \[e f_{i_0\dots i_p}-(-1)^pf_{i_0\dots i_p}e=-B_{dw}+\sum_{j=0}^p(-1)^jf_{i_0\dots \widehat{i_j}\dots i_p}.\] We will prove this by induction, but before we do, let us see how this proves the lemma.

Recall from the proof of lemma \ref{lem:connection} that after splitting $\At(\E)$ with respect to $\nabla_{i_0}$ the differential is given as the sum of three components
\[\gamma e=(-1)^q\left(\begin{matrix} 1\otimes e&[\nabla_{i_0},e]\\ 0& e\end{matrix}\right): \Omega^{\otimes q+1}\otimes \E_{i}\oplus \Omega^{q}\otimes \E_{i}\to \Omega^{q+1}\otimes \E_{i+1}\oplus \Omega^q\otimes \E_{i+1}\]
\[-\gamma m=(-1)^q\left(\begin{matrix} 0&-1\\ 0&0\end{matrix}\right):\Omega^{q+1}\otimes \E_i\oplus \Omega^{q}\otimes \E_i\to \Omega^{q}\otimes \E_i\oplus \Omega^{q-1}\otimes \E_i\]
and
\[B_{dw}=\left(\begin{matrix} B'_{dw}&0\\ 0&B''_{dw} \end{matrix}\right): \Omega^{q+1}\otimes \E_i \oplus \Omega^{q}\otimes \E_i\to \Omega^{q+2}\otimes \E_i \oplus \Omega^{q+1}\otimes \E_i.\]

We have the relation
\begin{footnotesize}
\begin{align*}
[\nabla_{i_0},&e]f_{i_0\dots i_p}-f_{i_0\dots i_p}\\
&=-\sum_{{k_1\dots k_p}}(-1)^{\sum_{j=1}j(k_j+1)}(\nabla_{i_0}-\nabla_{i_1})[\nabla_{i_1},e]^{k_1}\dots [\nabla_{i_p},e]^{k_p}\\
&=(\nabla_{i_1}-\nabla_{i_0})\sum_{{k_1\dots k_p}}(-1)^{k_1+\dots +k_p+p-1}(-1)^{\sum_{j=0}j(k_j+1)}(\nabla_{i_0}-\nabla_{i_1})[\nabla_{i_1},e]^{k_1}\dots [\nabla_{i_p},e]^{k_p}
\end{align*}
\end{footnotesize}
So when we take into account the grading operator $\gamma$ we get
\[(\gamma [\nabla_{i_0},e]-\gamma )f_{i_0\dots i_p}=(\nabla_{i_1}-\nabla_{i_0})f_{i_1\dots i_p}.\]
Now the observation is that $(\nabla_{i_1}-\nabla_{i_0})f_{i_1\dots i_p}$ is exactly the difference between splitting with respect to $\nabla_{i_1}$ and splitting with respect to $\nabla_{i_0}$ i.e. 
\[\sigma_{i_0}f_{i_1\dots i_p}+(\nabla_{i_1}-\nabla_{i_0})f_{i_1\dots i_p}=\sigma_{i_1}f_{i_1\dots i_p}\]
This combined with the claim gives us the lemma.

To prove the claim, notice first that we can write
\[f_{i_0\dots i_{p}}=\sum_{k=0}^n (-1)^{kp+p}f_{i_0\dots i_{p-1}}(\nabla_{i_{p-1}}-\nabla_{i_p})[\nabla_{i_p},d]^k.\]

Then we make the computation:
\begin{footnotesize}
\begin{align}\nonumber
\gamma e f_{i_0\dots i_p}-(-1)^pf_{i_0\dots i_p}e&=\sum_{k} (-1)^{kp+p+k+1}(\gamma e f_{i_0\dots i_{p-1}}-(-1)^{p-1}f_{i_0\dots i_{p-1}})(\nabla_{i_{p-1}}-\nabla_{i_p})[\nabla_{i_p},e]^k\\\nonumber
&\quad +\sum_{k} (-1)^{kp+k} f_{i_0\dots i_{p-1}}([\nabla_{i_p},e]-[\nabla_{i_{p-1}},e])[\nabla_{i_p},e]^k\\ \nonumber
&\quad +\sum_{k} (-1)^{kp} f_{i_0\dots i_{p-1}}(\nabla_{i_{p-1}}-\nabla_{i_p})((-1)^ke[\nabla_{i_p},e]^k-[\nabla_{i_p},e]^ke]\\ 
& \label{line7}=\sum_{k} (-1)^{kp+k+p+1}(-B_{dw}f_{i_0\dots i_{p-1}})(\nabla_{i_{p-1}}-\nabla_{i_p})[\nabla_{i_p},e]^k\\ 
&\label{line8}\quad +\sum_{k} (-1)^{(k+1)(p-1)}(\sum_{j=0}^{p-1}(-1)^{j}f_{i_0\dots \widehat{i_j}\dots i_{p-1}})(\nabla_{i_{p-1}}-\nabla_{i_p})[\nabla_{i_p},e]^k\\
&\label{line9}\quad +\sum_{k} (-1)^{kp+k} f_{i_0\dots i_{p-1}}([\nabla_{i_p},e]-[\nabla_{i_{p-1}},e])[\nabla_{i_p},e]^k\\
&\label{line10}\quad +\sum_{k} (-1)^{kp} f_{i_0\dots i_{p-1}}(\nabla_{i_{p-1}}-\nabla_{i_p})(-B_{dw}[\nabla_{i_p},e]^{k-1})
\end{align}
\end{footnotesize}

Note that the sign $(-1)^{kp+k+p+1}$ on the first line appears because applying 
$(\nabla_{i_{p-1}}-\nabla_{i_p})[\nabla_{i_p},e]^k$ before $f_{i_0\dots i_{p-1}}$ introduces an extra $k+1$ tensor factors of $\Omega^1$. So what we will need to show is that lines \eqref{line7}, \eqref{line8}, \eqref{line9} and \eqref{line10} sum to give
\[-B_{dw}f_{i_0\dots i_p}+\sum_{j=0}^p(-1)^jf_{i_0\dots \widehat{i_j}\dots i_p}.\]

Now we note that
\begin{small}
\begin{align*}
(-1)^{k+1}(-B_{dw}&f_{i_0\dots i_{p-1}})(\nabla_{i_{p-1}}-\nabla_{i_p})[\nabla_{i_p},e]^k\\
&+f_{i_0\dots i_{p-1}}(\nabla_{i_{p-1}}-\nabla_{i_p})(-B_{dw}[\nabla_{i_p},e]^{k-1})\\
&\quad\quad \quad\quad=B_{dw} f_{i_0\dots i_{p-1}}(\nabla_{i_{p-1}}-\nabla_{i_p})[\nabla_{i_p},e]^k
\end{align*}
\end{small}
so the sums 
\[\sum_{k} (-1)^{kp} f_{i_0\dots i_{p-1}}(\nabla_{i_{p-1}}-\nabla_{i_p})(-B_{dw}[\nabla_{i_p},e]^{k-1})\]
from \eqref{line10}
and
\[\sum_{k=0}^\infty (-1)^{kp+k+p+1}(-B_{dw}f_{i_0\dots i_{p-1}})(\nabla_{i_{p-1}}-\nabla_{i_p})[\nabla_{i_p},e]^k\]
from \eqref{line7} add to give $-B_{dw}f_{i_0\dots i_p}$ as needed. 

We have the following relation
\[f_{i_0\dots i_{p-1}}[\nabla_{i_{p-1}},e]=(-1)^{p-1}f_{i_0\dots i_{p-1}}-f_{i_0\dots i_{p-2}}(\nabla_{i_{p-2}}-\nabla_{i_{p-1}})\]
so then for the sum from \eqref{line9} we have
\begin{footnotesize}
\begin{align}\nonumber
\sum_{k} (-1)^{kp+k} f_{i_0\dots i_{p-1}}([\nabla_{i_p},e]&-[\nabla_{i_{p-1}},e])[\nabla_{i_p},e]^k\\
&= \nonumber\sum_{k} (-1)^{kp+k} f_{i_0\dots i_{p-1}}[\nabla_{i_p},e]^{k+1}\\
&\nonumber\quad -\sum_{k}(-1)^{kp+k} f_{i_0\dots i_{p-1}}[\nabla_{i_{p-1}},e][\nabla_{i_p},e]^{k}\\
&\nonumber=\sum_{k} (-1)^{kp+k} f_{i_0\dots i_{p-1}}[\nabla_{i_p},e]^{k+1}\\
&\nonumber\quad +\sum_{k}(-1)^{kp+k+p} f_{i_0\dots i_{p-1}}[\nabla_{i_p},e]^{k}\\
&\nonumber\quad +\sum_{k}(-1)^{kp+k+1} f_{i_0\dots i_{p-2}}(\nabla_{i_{p-2}}-\nabla_{i_{p-1}})[\nabla_{i_p},e]^{k}\\
&\label{line11}=(-1)^pf_{i_0\dots i_{p-1}}\\
&\label{line12}\quad +\sum_{k}(-1)^{kp+k+1} f_{i_0\dots i_{p-2}}(\nabla_{i_{p-2}}-\nabla_{i_{p-1}})[\nabla_{i_p},e]^{k}
\end{align}
\end{footnotesize}

Now we turn our attention to the sum 
\[\sum_{k=0}^n(-1)^{(k+1)(p-1)}(\sum_{j=0}^{p-1}(-1)^{j}f_{i_0\dots \widehat{i_j}\dots i_{p-1}})(\nabla_{i_{p-1}}-\nabla_{i_p})[\nabla_{i_p},e]^k\]
from \eqref{line8}
In the case when $j\neq p-1$ we have
\begin{equation} \label{eqn:jneqp-1}\sum_{k=0}^\infty (-1)^{(k+1)(p-1)}f_{i_0\dots \widehat{i_j}\dots_{i_{p-1}}}(\nabla_{i_{p-1}}-\nabla_{i_p})[\nabla_{i_p},e]^{k}=f_{i_0\dots \widehat{i_j}\dots i_p}
\end{equation}

When $j=p-1$ we may add 
\[\sum_{k=0}^n (-1)^{(k+1)(p-1)+p-1} f_{i_0\dots i_{p-2}}(\nabla_{i_{p-1}-\nabla_{i_p}})[\nabla_{i_p},e]^k\]
to 
\[\sum_{k=0}^n (-1)^{kp+k+1}f_{i_0\dots i_{p-2}}(\nabla_{i_{p-2}}-\nabla_{i_{p-1}})[\nabla_{i_p},e]^k\]
from \eqref{line12} to get
\begin{equation} \label{eqn:jeqp-1}(-1)^{p-1}\sum_{k=0}^\infty (-1)^{k(p-1)} f_{i_0\dots i_{p-2}}(\nabla_{i_{p-2}}-\nabla_{i_p})[\nabla_{i_p},e]^k=(-1)^{p-1}f_{i_0\dots \widehat{i_{p-1}} i_p}.
\end{equation}
Adding the sums from \eqref{line11}, \eqref{eqn:jneqp-1} and \eqref{eqn:jeqp-1} gives
\[\sum_{j=0}^p(-1)^jf_{i_0\dots \widehat{i_j}\dots i_p}\]
which finishes the claim and thus the lemma. \\

\end{proof}

\begin{thm} \label{thm:formula} Let $f=\{f_{i_0\dots i_p}\}\in Cech(\sHom(\E,\E))$. The boundary bulk map $\tau_\E$ is computed on $f$ as
\begin{footnotesize}
\[f_{i'_0\dots i'_q}\mapsto str\left(\sum_{p}\sum_{i_0\dots i_p}\sum_{k_1\dots k_p} (-1)^{p+\sum_{j=0}^p jk_j} \frac{[\nabla_{i_0},e]^{k_0}(\nabla_{i_0}-\nabla_{i_1})\dots (\nabla_{i_{p-1}}-\nabla_{i_p})[\nabla_{i_p},e]^{k_p}}{(k_0+\dots+k_p+p)!}\circ f_{i'_0\dots i'_q}\right)\]
\end{footnotesize}
\end{thm}
\begin{proof}
This is mostly an amalgamation of lemmas \ref{lem:gluing}, \ref{lem:maps} and \ref{lem:sheafTrace}. The division by $(k_0+\dots k_p+p)!$ comes from applying the map 
\[\sum \frac{\wedge}{i!}:Cech(\At(\E))\to Cech(\Omega_{dw}\otimes \E)\] (see lemma \ref{lem:wedgeq}). The sign comes from the fact that
\[\frac{p(p-1)}{2}+\sum_{j=0}^pj(k_j+1)=\frac{p(p-1)}{2}+\frac{p(p+1)}{2}+\sum_{j=0}^pjk_j=p^2+\sum_{j=0}^pjk_j\]
and $p$ is congruent to $p^2$ modulo 2. We need only check that this map we have constructed actually computes the boundary bulk-map. 

We have the following diagram
\begin{center}
\begin{tikzpicture}
\matrix(m)[matrix of math nodes, row sep=3em, column sep=3em]{\At(\E)&Cech(\At(\E))\\ \E&Cech(\E)\\};

\path[->]

(m-1-1) edge (m-1-2)
(m-1-1) edge node[auto,swap]{$\pi$} (m-2-1)
(m-2-1) edge (m-1-2)
(m-2-1) edge (m-2-2)
(m-1-2) edge node[auto]{$Cech(\pi)$}(m-2-2);

\end{tikzpicture}
\end{center}
where the diagonal map is given by
\begin{equation} \label{eqn:diagonalMap} \sum_{p}\sum_{i_0\dots i_p}\sum_{k_1\dots k_p} (-1)^{p^2+\sum_{j=0}^p jk_j} [\nabla_{i_0},e]^{k_0}(\nabla_{i_0}-\nabla_{i_1})\dots (\nabla_{i_{p-1}}-\nabla_{i_p})[\nabla_{i_p},e]^{k_p}.
\end{equation}

Now the outside square commutes as well as the bottom right triangle. And then, since $Cech(\pi)$ is a weak equivalence, the upper left triangle commutes in the coderived category. It follows then that composition of the diagonal map, \eqref{eqn:diagonalMap}, with the map $\sum_{i} \frac{\wedge}{i!}:Cech(\At(\E))\to Cech(\Omega_{dw})$ computes $\iota\E xp(\At(\E))$, where $\iota: \Omega_{dw}\to Cech(\Omega_{dw})$ is the inclusion. 
\end{proof}

\begin{cor} \label{cor:formula} The Chern Character $\E$ is given by the cocycle 
\begin{small}
\[ch(\E)=str\left(\sum_{p}\sum_{i_0\dots i_p}\sum_{k_1\dots k_p} (-1)^{p+\sum_{j=0}^p jk_j} \frac{[\nabla_{i_0},e]^{k_0}(\nabla_{i_0}-\nabla_{i_1})\dots (\nabla_{i_{p-1}}-\nabla_{i_p})[\nabla_{i_p},e]^{k_p}}{(k_0+\dots+k_p+p)!}\right)\]
\end{small}
in a Cech model for $\mathbb R\Gamma(\Omega_{dw})$. 
\end{cor}

\begin{rmk} In light of remark \ref{rmk:sheaves} and theorem \ref{thm:coincide}, corollary \ref{cor:formula} translates directly to give a formula for the Chern character of complexes of vector bundles.
\end{rmk}

\bibliography{mf.bib}{}
\bibliographystyle{plain}
\end{document}